\documentclass[11pt]{amsart}
\usepackage{amsmath, amsthm, amssymb}
\usepackage{graphicx}
\usepackage{mathbbol}
\usepackage{txfonts}
\usepackage[usenames]{color}
\usepackage{xr}
\usepackage{tikz}
\usepackage{amsfonts}

\newtheorem{thm}{Theorem}[section]
\newcommand{\bt}{\begin{thm}}
\newcommand{\et}{\end{thm}}

\newtheorem{cor}[thm]{Corollary}   

\newcommand{\bc}{\begin{cor}}
\newcommand{\ec}{\end{cor}}

\newtheorem{lem}[thm]{Lemma}   

\newcommand{\bl}{\begin{lem}}
\newcommand{\el}{\end{lem}}

\newtheorem{prop}[thm]{Proposition}
\newcommand{\bp}{\begin{prop}}
\newcommand{\ep}{\end{prop}}

\newtheorem{defn}[thm]{Definition}
\newtheorem{conj}[thm]{Conjecture}

\newcommand{\bd}{\begin{defn}}    
\newcommand{\ed}{\end{defn}}

\newtheorem{rmrk}[thm]{Remark}   

\newcommand{\br}{\begin{rmrk}}
\newcommand{\er}{\end{rmrk}}

\newtheorem{example}[thm]{Example}

\newcommand{\be}{\begin{equation}}
\newcommand{\ee}{\end{equation}}

\newcommand{\C}{\mathbb{C}}

\newcommand{\R}{\mathbb{R}}

\newcommand{\E}{\mathbb{E}}

\newcommand{\Z}{\mathbb{Z}}

\newcommand{\diam}{\operatorname{diam}}

\newcommand{\Ricci}{\rm{Ricci}}

\newcommand{\Lip}{\operatorname{Lip}}

\newcommand{\vol}{\operatorname{Vol}}

\begin{document}

\title[Diameter Controls and Smooth Convergence away from Singular Sets]{Diameter Controls and Smooth Convergence away from Singular Sets}

\author{Sajjad Lakzian}
\thanks{Lakzian was partially supported
by NSF DMS \#1006059.}
\address{CUNY Graduate Center}
\email{SLakzian@gc.cuny.edu}

\date{}

\keywords{}

\begin{abstract}

We prove that if a family of metrics, $g_i$, on a compact Riemannian
manifold, $M^n$, have a uniform lower Ricci curvature bound and
converge to $g_\infty$ smoothly away from a singular set, $S$, with
Hausdorff measure, $H^{n-1}(S) = 0$, and if there exists connected
precompact exhaustion, $W_j$, of $M^n \setminus S$ satisfying
$\diam_{g_i}(M^n) \le D_0 $, $\vol_{g_i}(\partial W_j) \le A_0 $ and
$\vol_{g_i}(M^n \setminus W_j) \le V_j where \lim_{j\to\infty}V_j=0 $
then the Gromov-Hausdorff limit exists and agrees with the metric
completion of $(M^n \setminus S, g_\infty)$.    This is a strong
improvement over prior work of the author with
Sormani that had the additional assumption that the singular set had
to be a smooth submanifold of codimension two.  We have a second main
theorem in which the Hausdorff measure condition on $S$ is replaced by
diameter estimates on the connected components of the boundary of the
exhaustion, $\partial W_j$. This second theorem allows for singular
sets which are open subregions of the manifold.  In addition, we show
that the uniform lower Ricci curvature bounds in these theorems can be
replaced by the existence of a uniform linear contractibility
function. If this condition is removed altogether, then we prove that
$\lim_{j\to \infty} d_{\mathcal{F}}(M_j', N')=0$, in which $M_j'$ and
$N'$ are the settled completions of $(M, g_j)$ and $(M_\infty\setminus
S, g_\infty)$ respectively and $d_{\mathcal{F}}$ is the Sormani-Wenger
Intrinsic Flat distance. We present examples demonstrating the
necessity of many of the hypotheses in our theorems. Finally, as an application, we will prove the Candella-de la Ossa's conjecture for Calabi-Yau conifolds. 

\end{abstract}

\maketitle


\section{Introduction}

In this paper, we will prove results concerning the smooth convergence of Riemannian metrics away from a singular set $S$ and will provide some important application of our results. One definition of smooth convergence away from singularities is as follows:
 
\begin{defn} \label{defn-smoothly-1}
We will say that a sequence of Riemannian metrics $g_i$ on a
compact manifold $M^n$ converges smoothly away from $S \subset M^n$
to a Riemannian metric $g_\infty$ on $M^n \setminus S$ if
for every compact set $K\subset M^n\setminus S$, $g_i$
converge $C^{k,\alpha}$ smoothly to $g_\infty$ as tensors.
\end{defn}

Right away from the definition, it is apparent that the global geometry is
not well controlled under such convergence.   It is natural to ask under
what additional conditions the original sequence of manifolds, 
$M_i=(M^n, g_i)$ have the expected
Gromov-Hausdorff (GH) and Sormani-Wenger Intrinsic Flat (SWIF) limits 
\cite{Gromov-metric} \cite{SorWen2}.   Recall that there are
examples of sequences of metrics on spheres which converge
smoothly away from a point singularity which have no subsequence
converging in the GH or the SWIF sense, so additional
conditions are necessary (c.f. \cite{Lakzian-Sormani}). 

Many results concerning GH limits of the $M_i$ have appeared in the literature. For example, Anderson in \cite{Anderson-KE} studies the convergence of Einstein metrics to orbifolds.  Bando-Kasue-Nakajima in \cite{BKN} studies the singularities of the Einstein ALF manifolds.  Eyssidieux-Guedj-Zeriahi in \cite{EGZ} prove similar results for the solutions to the complex Monge-Ampere equation. Also Huang in \cite{Huang-convergence} , Ruan-Zhong in \cite{Ruan-Zhang}  , Sesum in \cite{Sesum-convergence}, Tian in \cite{Tian-surfaces} and Tosatti in \cite{Tosatti-KE} study the convergence of Kahler-Einstein metrics and Kahler-Einstein orbifolds. However, even in this setting, the relationship is not completely clear and the limits need not agree (see \cite{Bando-bubbling}.) In Tian-Viaclovsky \cite{Tian-Viac}, compactness results for various classes Riemannian metrics in dimension four were obtained in particular for anti-self-dual metrics, Kahler metrics with constant scalar curvature, and metrics with harmonic curvature. Also the relation between different notions of convergence for Ricci flow is studied in ~\cite{RF_SWIF}.

Here we first study SWIF limits of sequences of manifolds which
converge away from
a singular set and then prove the SWIF and GH limits agree using
techniques developed in
prior work of the author with Sormani in [LS].   All necessary
background on these
techniques and on SWIF convergence is reviewed in Section 2.

\begin{thm}\label{codim-thm}
Let $M_i=(M^n , g_i)$ be a sequence of compact oriented Riemannian manifolds
such that
there is a subset, $S$, with $H^{n-1}(S) = 0$
and connected precompact exhaustion,
$W_j$, of $M\setminus S$ satisfying (\ref{defn-precompact-exhaustion})
with $g_i$ converge smoothly to $g_\infty$ on each $W_j$,
\be\label{m-diam}
\diam_{M_i}(W_j) \le D_0 \qquad \forall i\ge j, 
\ee
\be \label{m-area}
\vol_{g_i}(\partial W_j) \le A_0,
\ee
and
\be \label{m-edge-volume}
\vol_{g_i}(M\setminus W_j) \le V_j \textrm{ where } \lim_{j\to\infty}V_j=0.
\ee
Then
\be
\lim_{i\to \infty} d_{\mathcal{F}}(M_i', N')=0.
\ee
where  $M_i'$ and
$N'$ are the settled completion
of $(M, g_i)$ and $(M\setminus S, g_\infty)$ respectively.
\end{thm}

Here $\diam_{M}(W)$ is the extrinsic diameter found by
\be \label{ext-diam}
\diam_{M}(W)=\sup \{d_M(x,y): \, x,y \in W\}
\ee
where $d_M$ is the extrinsic distance measured in $M$ rather than $W$:
\be \label{ext-dist}
d_M(x,y)=\inf\{L(C): \, C:[0,1]\to M, \,\, C(0)=x,\,\, C(1)=y\}.
\ee
We write $M_i=(M, g_i)$.  The intrinsic diameter of $W$
is then $\diam_W(W)$.
See Remark \ref{codim-necessity} for the necessity of the hypotheses in Theorem \ref{codim-thm}.  

Under the conditions of Theorem~\ref{codim-thm}, if we assume in addition that
the manifolds in the
sequence have a uniform lower bound on Ricci curvature, then the SWIF
and GH limits
agree, so we obtain the following new theorem relating the GH limit to
the metric completion of
the smooth limit away from the singularity:

\begin{thm}\label{Ricci-thm-improved}
Let $M_i=(M,g_i)$ be a sequence of oriented compact Riemannian manifolds
with uniform lower Ricci curvature bounds, 
\be
\Ricci_{g_i}(V,V)\ge (n-1)H \, g_i(V,V) \qquad \forall V \in TM_i  \;,
\ee
which converges smoothly away from
a singular set, $S$, with $H^{n-1}(S) = 0$. 
If there is a connected precompact exhaustion, $W_j$, of
$M\setminus S$,
\be \label{defn-precompact-exhaustion}
\bar{W}_j \subset W_{j+1} \textrm{ with } 
\bigcup_{j=1}^\infty W_j=M\setminus S,
\ee
satisfying 
\be \label{diam-2}
\diam(M_i) \le D_0,
\ee
\be \label{area-2}
\vol_{g_i}(\partial W_j) \le A_0,
\ee
and
\be \label{not-vol-2}
\vol_{g_i}(M\setminus W_j) \le V_j \textrm{ where } \lim_{j\to\infty}V_j=0,
\ee
then
\be
\lim_{j\to \infty} d_{GH}(M_j, N)=0,
\ee
where $N$ is the
metric completion of $(M\setminus S, g_\infty)$.
\end{thm}

See Remark \ref{codim-necessity} for the necessity of our hypotheses in Theorem \ref{Ricci-thm-improved}. We may replace the Ricci condition by a condition on contractibility (see Theorem~\ref{c-codim-thm} .) For the necessity of the hypotheses in this theorem see \cite[Remark 6.8]{Lakzian-Sormani}.

Theorems~\ref{codim-thm}-~\ref{Ricci-thm-improved} improve upon a prior result of
the author and Sormani in [LS] because we no longer require the
singular set to be a smooth
submanifold of codimension 2 as was required there.    In fact, we may
even allow the singular
set to be an open domain as long as we have sufficiently strong
controls on the diameters of the
exhaustion's boundaries as seen in the following theorem:

\begin{thm}\label{diam-thm}
Let $M_i=(M,g_i)$ be a sequence of Riemannian manifolds
such that
there is a closed subset, $S$, and a connected precompact exhaustion,
$W_j$, of $M\setminus S$ satisfying (\ref{defn-precompact-exhaustion})
such that $g_i$ converge smoothly to $g_\infty$ on each $W_j$.

If each connected component of $M\setminus W_j$ has a connected boundary,
\be \label{new-bridge-2}
\limsup_{i \to \infty} \;  \left\{ \sum_\beta \diam_{(\Omega^\beta_j , g_i)} (\Omega^\beta_j) :  \text{$\Omega^\beta_j$ connected component of $\partial W_j$} \right\} \;
\le B_j,
\ee
where $\lim_{j \to \infty} B_j = 0$, and if we have
\be \label{m-int-diam}
\diam_{(W_j, g_i)}(W_j)\le D_{int}, 
\ee
\be 
\vol_{g_i}(\partial W_j) \le A_0,
\ee
\be \label{m-edge-volume}
\vol_{g_i}(M\setminus W_j) \le V_j \textrm{ where } \lim_{j\to\infty}V_j=0,
\ee
then
\be
\lim_{j\to \infty} d_{\mathcal{F}}(M_j', N')=0.
\ee
where $N'$ is the settled completion
of $(M\setminus S, g_\infty)$.
\end{thm}

See Remark \ref{necessity-diam} for the necessity of our hypotheses in Theorem \ref{diam-thm}.

In presence of a uniform lower Ricci curvature bound, Theorem \ref{diam-thm} can be applied to prove the following theorem:

\begin{thm}\label{Ricci-diam-thm}
Let $M_i=(M,g_i)$ be a sequence of oriented Riemannian manifolds
with uniform lower Ricci curvature bounds, 
\be
\Ricci_{g_i}(V,V)\ge (n-1)H \, g_i(V,V) \qquad \forall V \in TM_i \; ,
\ee
which converges smoothly away from
a closed singular set, $S$.

If there is a connected precompact exhaustion, $W_j$, of
$M\setminus S$, satisfying (\ref{defn-precompact-exhaustion})
such that 
each connected component of $M\setminus W_j$ has a connected boundary,
\be \label{new-bridge-2}
\limsup_{i \to \infty} \;  \left\{ \sum_\beta \diam_{(\Omega^\beta_j , g_i)} (\Omega^\beta_j) :  \text{$\Omega^\beta_j$ connected component of $\partial W_j$} \right\} \;
\le B_j,
\ee
where $\lim_{j \to \infty} B_j = 0$, and if we have
\be \label{diam-M}
	\diam(M_i)\le D_0, 
\ee
\be \label{m-int-diam}
\diam_{(W_j, g_i)}(W_j)\le D_{int}, 
\ee
and volume controls:
\be \label{area-2}
\vol_{g_i}(\partial W_j) \le A_0,
\ee
\be \label{not-vol-2}
\vol_{g_i}(M\setminus W_j) \le V_j \textrm{ such that } \lim_{j\to\infty}V_j=0,
\ee
then
\be
\lim_{j\to \infty} d_{GH}(M_j, N)=0.
\ee
where $N$ is the metric completion
of $(M\setminus S, g_\infty)$.
\end{thm}

In Theorems \ref{Ricci-thm-improved} and \ref{Ricci-diam-thm}, the diameter hypothesis $\diam(M_i) \le D_0$
is not necessary when the Ricci curvature is nonnegative (see Lemma~\ref{lem-Ricci-diam}. )

The Ricci curvature condition in Theorems \ref{Ricci-thm-improved} and \ref{Ricci-diam-thm} may be replaced by a requirement that the sequence of manifolds have a uniform linear contractibility function (see Theorem~\ref{c-diam-GH} and Theorem \ref{c-codim-thm}. )
See Definition~\ref{defn-contractibility-function} for the definition of a contractibility function.   
Recall that Greene-Petersen have a compactness theorem for
sequences of manifolds with uniform contractibility functions and upper
bounds on their volume \cite{Greene-Petersen}.
 
Subsequently, as an application of our Theorems and relying on diameter bounds and convergence results obtained by Rong-Zhang~\cite{Rong-Zhang} and Tossati~\cite{Tosatti-KE}, we will prove the following Theorem which is known in the literature as the Candella-de la Ossa's conjecture, for Calabi-Yau conifolds. 
 
 \begin{thm}\label{thm-Candelas}
 
 Let $M_0$ be a singular $n-$dimensional normal variety with isolated \emph{conifold} singularities then,\\
	
	\textbf{(i)} Extremal transitions $\bar{M} \to M_0 \leadsto M_t   \;\;\;  (t \neq 0)$ are continuous with respect to the Gromov-Hausdorff distance i.e. there exist families of Ricci-flat K\"{a}ler metics $\bar{g}_s$ on $\bar{M}$ and $g_t$ on $M_t$ and compact metric space $\left( X , d_X \right)$ such that
	\be
		\left( \bar{M} , \bar{g}_s \right) \xrightarrow{G-H} \left( X , d_X \right)\xleftarrow{G-H} \left( M_t , g_t\right)
	\ee\\
	
	\textbf{(ii)}  Flops $\bar{M}_1 \to M_0 \dashrightarrow \bar{M}_2$ are continuous with respect to the Gromov-Hausdorff distance i.e. there families of Ricci-flat K\"{a}hler metrics $\bar{g}_{i,s}$ and a compact metric space $\left( X , d_X \right) $ such that
	
	\be
		\left( \bar{M}_1 , \bar{g}_{1,s} \right) \xrightarrow{G-H} \left( X , d_X \right)\xleftarrow{G-H} \left( \bar{M}_2 , \bar{g}_{2,s}\right)
	\ee
 
\end{thm}

 \begin{rmrk}
 Different proofs of Candella-de la Ossa's conjecture in various settings has been given in Rong-Zhang~\cite{Rong-Zhang} and Song~\cite{Song-Candelas-conj}. Our proof is a more coarse geometric proof and to a great deal avoids the advanced PDE techniques used in other proofs.
\end{rmrk}

The paper is organized as follows: In Section~\ref{Sect-review}, we will briefly review all the notions and theorems that we have used in this paper; In Section~\ref{main-result-I}, we give a proof of our
theorems which assume $\mathcal{H}^{n-1}(S)=0$ [Theorem~\ref{codim-thm},Theorem~\ref{Ricci-thm-improved} and Theorem~\ref{c-codim-thm}]; Section~\ref{main-result-II} is devoted to the proof of our
theorems which replace the Hausdorff measure hypothesis with diameter bounds [Theorem~\ref{diam-thm}, Theorem~\ref{Ricci-diam-thm} and Theorem~\ref{c-diam-GH}]; Section~\ref{Sect-examples} discusses a few interesting examples to illustrate the underlying phenomena and also to prove the necessity of some our hypotheses and finally, In Section \ref{sec-applications}, we present a proof of Theorem \ref{thm-Candelas}.

\section*{Acknowledgements}The author would like to thank Christina Sormani for her constant encouragement, assistance with exposition and particularly for sketching ideas towards Examples~\ref{ex-pinched-torus} and~\ref{ex-slit-torus}. Also many thanks go to Xiaochun Rong for teaching the author about the Candella-de la Ossa's conjecture.
He would like to thank
Dimitri Burago specifically for drawing Example~\ref{ex-pulled-torus} in conversations with his advisor that were then related to him.
The author also would like to express his gratitude to Jeff Cheeger and Fanghua Lin for teaching him advanced Riemannian Geometry and Geometric Measure Theory. The author is very grateful to Karl-Theodor Sturm for his interest in the author's research and for the opportunity to work with him as a postdoctoral fellow at the Hausdorff Center for Mathematics in Bonn.


\section{Background} \label{Sect-review}

\subsection{Metric Completion and Settled Completion}

We give a very brief review of the definitions of metric completion and settled completion of a metric space:

\begin{defn}\label{metric-completion}
Given a precompact metric space, $X$, the metric completion, $\bar{X}$
of $X$ is the space of Cauchy sequences, $\{x_j\}$, in $X$ with the
metric
\be
d(\{x_j\},\{y_j\})=\lim_{j\to\infty} d_X(x_j,y_j),
\ee
and where two Cauchy sequences are identified if the distance between
them is $0$.  There is an isometric embedding, $\varphi: X \to \bar{X}$,
defined by $\varphi(x)=\{x\}$ where $\{x\}$ is a constant sequence.
Lipschitz functions, $F: X\to Y$, extend to $F:\bar{X} \to Y$ via
$F(\{x_j\})=\lim_{j\to\infty} F(x_j)$ as long as $Y$ is complete.
\end{defn}

The definition of the settled completion which is essential in studying the Intrinsic Flat convergence of metric spaces is as follows

\begin{defn}[Sormani-Wenger \cite{SorWen2}] \label{defn-positive-density}
The settled completion, $X'$, of a metric space $X$ with a measure $\mu$
is the collection of points $x$ in the metric completion $\bar X$
which have positive  lower density
\be\label{eq-positive-density}
\liminf_{r\to 0} \mu(B_p(r))/r^m >0.
\ee
The resulting space is then ``completely settled". 
\end{defn}

\subsection{GH and SWIF Distances}

The Gromov-Hausdorff (GH) distance was defined by Gromov as 
\be
d_{GH}(X_1, X_2)= \inf\left\{ d^Z_H(\varphi_1(X_1), \varphi_2(X_2)): \,\,\varphi_i: X_i \to Z\right\}
\ee
where the infimum is taken over all common metric spaces, $Z$,
and all isometric embeddings, $\varphi_i: X_i \to Z$.
Limit spaces obtained from the Gromov-Hausdorff convergence are compact metric spaces \cite{Gromov-metric}. 

The Sormani Wenger Intrinsic Flat (SWIF) distance is defined similarly by replacing the Hausdorff distance by the flat distance, viewing $\varphi_i(M_i)$ as integral current spaces in the sense of \cite{AK}.
The limit spaces obtained under intrinsic flat convergence are ``integral current spaces": completely settled metric spaces with an integral current structure that defines a notion of integration over $m$-forms. When the limit is the settled completion of an open manifold, this the
integral current structure is simply defined by integration over that open manifold (see ~\cite{SorWen2} for more details.)

The SWIF distance, $d_{\mathcal{F}}(M_1, M_2)$, is estimated by explicitly constructing a filling manifold, $B^{m+1}$,
between the two given manifolds, finding the excess boundary
manifold $A^m$ satisfying
\be\label{Stokes}
\int_{\varphi_1(M_1)}\omega -\int_{\varphi_2(M_2)}\omega=\int_B d\omega +\int_A \omega,
\ee
 and summing their volumes
\be \label{est-int-flat}
d_{\mathcal{F}}(M^m_1, M^m_2) \le \vol_m(A^m)+\vol_{m+1}(B^{m+1}). 
\ee
In the next subsection we present review a theorem which clarifies
the concept of the intrinsic flat distance while proving a means of
estimating it.

\subsection{Estimating the Gromov Hausdorff and Intrinsic Flat Distance.}
We can estimate both of these distances by applying the following theorem which was proven in prior work of the author with Sormani ~\cite{Lakzian-Sormani} by constructing an explicit space $Z$ and isometric embeddings $\varphi_i$.  Here we have cut and pasted the exact
theorem statement along with the corresponding figure from
that paper:   

\begin{thm} \label{thm-subdiffeo}
Suppose $M_1=(M,g_1)$ and $M_2=(M,g_2)$ are oriented
precompact Riemannian manifolds
with diffeomorphic subregions $U_i \subset M_i$ and
diffeomorphisms $\psi_i: U \to U_i$ such that
\be \label{thm-subdiffeo-1}
\psi_1^*g_1(V,V)
< (1+\epsilon)^2 \psi_2^*g_2(V,V) \qquad \forall \, V \in TU,
\ee
and
\be \label{thm-subdiffeo-2}
\psi_2^*g_2(V,V) <
(1+\epsilon)^2 \psi_1^*g_1(V,V) \qquad \forall \, V \in TU.
\ee
Taking the extrinsic diameters,
\be \label{DU}
D_{U_i}= \sup\{\diam_{M_i}(W): \, W\textrm{ is a connected component of } U_i\} \le \diam(M_i).
\ee
we define a hemispherical width,
\be \label{thm-subdiffeo-3}
a>\frac{\arccos(1+\epsilon)^{-1} }{\pi}\max\{D_{U_1}, D_{U_2}\}.
\ee
Taking the difference in distances with respect to the outside manifolds,
\be \label{lambda}
\lambda=\sup_{x,y \in U}
|d_{M_1}(\psi_1(x),\psi_1(y))-d_{M_2}(\psi_2(x),\psi_2(y))|,
\ee
we define heights,
\be \label{thm-subdiffeo-4}
h =\sqrt{\lambda ( \max\{D_{U_1},D_{U_2}\} +\lambda/4 )\,},
\ee
and
\be \label{thm-subdiffeo-5}
\bar{h}= \max\{h,  \sqrt{\epsilon^2 + 2\epsilon} \; D_{U_1}, \sqrt{\epsilon^2 + 2\epsilon} \; D_{U_2} \}.
\ee
Then the Gromov-Hausdorff distance \emph{ the metric
completions} is bounded,
\be \label{thm-subdiffeo-6}
d_{GH}(\bar{M}_1, \bar{M}_2 ) \le a + 2\bar{h} +
\max\left\{ d^{M_1}_H(U_1, M_1), d^{M_2}_H(U_2, M_2)\right\},
\ee
and the intrinsic flat distance between the
\\emph{settled completions} is bounded,
\begin{eqnarray*}
d_{\mathcal{F}}(M'_1, M'_2) &\le&
\left(2\bar{h} + a\right) \Big(
\vol_m(U_{1})+\vol_m(U_2)+\vol_{m-1}(\partial U_{1})+\vol_{m-1}(\partial U_{2})\Big)\\
&&+\vol_m(M_1\setminus U_1)+\vol_m(M_2\setminus U_2),
\end {eqnarray*}

\end{thm}

\begin{figure}[h] 
   \centering
   \includegraphics[width=5in]{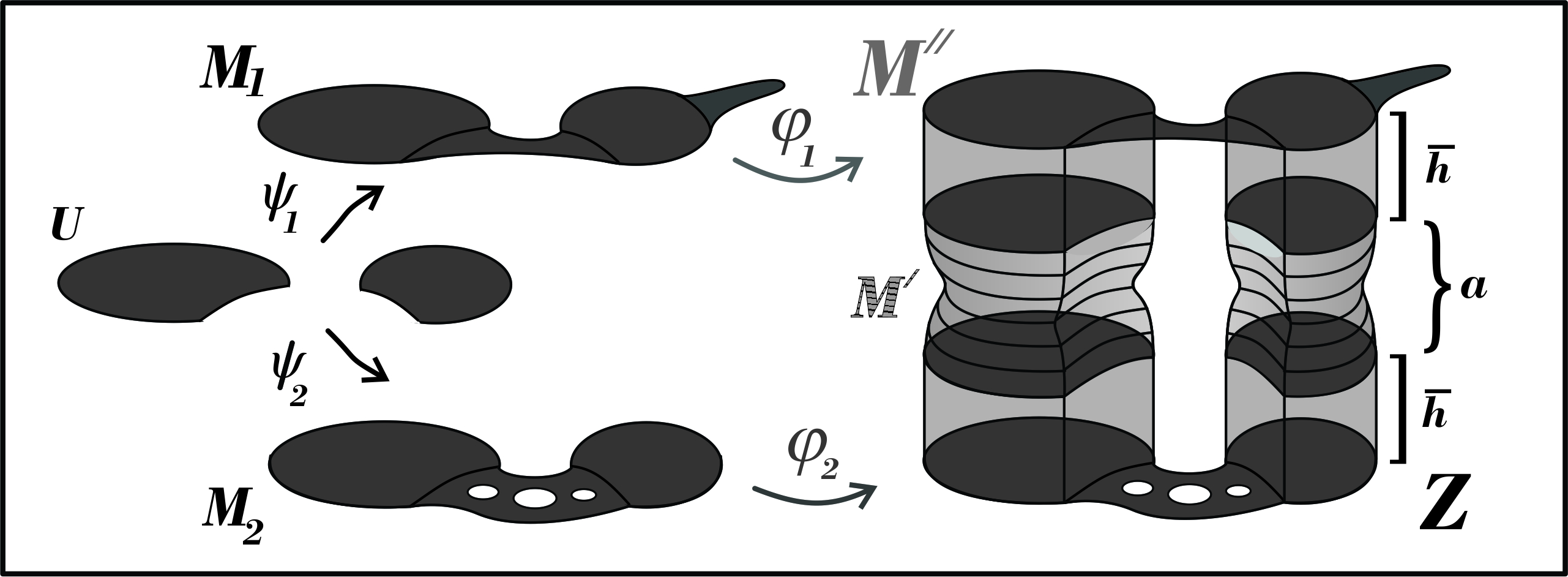}
   \caption{Creating $Z$ for Theorem~\ref{thm-subdiffeo}.}
   \label{fig-subdiffeo-1}
\end{figure}

Note that permission to reprint this figure along with the
statement of Theorem~\ref{thm-subdiffeo} has been granted
by the author and Christina Sormani who own the copyright to this figure
that first appeared in \cite{Lakzian-Sormani}.

\subsection{Review of Compactness Theorems}

\begin{defn}\label{defn-contractibility-function}
A function $\rho:[0,r_0]\to [0,\infty)$ is a contractibility
function for a manifold $M$ with metric $g$ if 
every ball $B_p(r)$ is contractible within $B_p(\rho(r))$.
\end{defn}

\begin{thm} [Gromov \cite{Gromov-metric}] 
A sequence of compact Riemannian manifolds, $(M_j,g_j)$,
such that $\diam(M_j) \le D$ and $Ricci_{M_j} \ge -H$, has
a subsequence converging in the Gromov-Hausdorff
sense to a metric space $(X,d)$. 
\end{thm}

\begin{thm}[Greene-Petersen\cite{Greene-Petersen}] 
A sequence of compact Riemannian manifolds, $(M_j,g_j)$,
such that $\vol(M_j) \le V$ and such that
there is a uniform contractibility
function,  $\rho:[0,r_0]\to [0,\infty)$, for all the $M_j$, has
a subsequence converging in the Gromov-Hausdorff
sense to a metric space $(X,d)$. 
\end{thm}

In \cite{SorWen1} the following theorems were proven
which can be applied to deduce information about the Gromov-Hausdorff
limit of a sequence.

\begin{thm}[Sormani-Wenger \cite{SorWen2}] \label{thm-sw-contractible}
If a sequence of oriented compact Riemannian manifolds, $(M_j,g_j)$,
with a uniform linear contractibility
function, $\rho: [0,\infty)\to [0,\infty)$
and a uniform upper bound on volume, $\vol(M_j)\le V$,
converges in the Gromov-Hausdorff
sense to $(X,d)$, then it converges in the intrinsic flat Sense
to $(X,d,T)$ (see Theorem 4.14 of \cite{SorWen2}).
\end{thm}

Recall that, in general, the intrinsic flat limits and Gromov-Hausdorff limits
need not agree \cite[Examples 2.3 and 2.3]{Lakzian-Sormani} because
intrinsic flat limits do not include points with $0$ density as in (\ref{eq-positive-density}).
In fact intrinsic flat limits may exist when Gromov-Hausdorff limits do not
[Example~\ref{ex-many-tips-scalar}].

\begin{thm}[Sormani-Wenger] \label{thm-sw-Ricci}
If a sequence of oriented compact Riemannian manifolds, $(M_j,g_j)$,
such that $\diam(M_j) \le D$ and $Ricci_{M_j} \ge 0$
and $vol(M_j) \ge V_0$ converges in the Gromov-Hausdorff
sense to $(X,d)$, then it converges in the intrinsic flat Sense
to $(X,d,T)$  (see Theorem 4.16 of \cite{SorWen2}).
\end{thm}

This theorem is conjectured to hold with uniform lower bounds
on Ricci curvature \cite{SorWen2}. 

\subsection{Review of Smooth Convergence away from Singular Sets.}
Here we review results from our prior work with Sormani in \cite{Lakzian-Sormani}.   
We need the following key definition (a bound on the metric distortions) in order to state the results. Recall the definition of the
extrinsic distance in (\ref{ext-dist}).

\begin{defn} \label{well-embedded}
Given a sequence of Riemannian manifolds
$M_i=(M,g_i)$ and an open subset, $U\subset M$,
a connected precompact exhaustion, $W_j$, of $U$ 
satisfying (\ref{defn-precompact-exhaustion})
is
\emph{uniformly well embedded} if there exist a $\lambda_0$ such that
\be \label{lambda-ijk-00}
\limsup_{j\to \infty} \limsup_{k\to \infty} \limsup_{i\to\infty} \lambda_{i,j,k}
\le \lambda_0,
\ee
and
\be\label{lambda-ijk-2}
\limsup_{k\to\infty} \lambda_{i,j,k}=\lambda_{i,j}
\textrm{ where } \limsup_{i\to \infty} \lambda_{i,j}=\lambda_j
\textrm{ and }\lim_{j\to\infty} \lambda_j=0.
\ee
\end{defn}
where,
\be\label{lambda-ijk}
\lambda_{i,j,k}= \sup_{x,y\in W_j} |d_{(W_{k}, g_i)}(x,y)- d_{(M,g_i)}(x,y).
\ee

The author and Sormani in ~\cite{Lakzian-Sormani} have proven:

\begin{thm}\label{flat-to-settled}
Let $M_i=(M,g_i)$ be a sequence of Riemannian manifolds
such that
there is a closed subset, $S$, and a 
uniformly well embedded connected precompact exhaustion,
$W_j$, of $M\setminus S$ satisfying (\ref{defn-precompact-exhaustion})
such that $g_i$ converge smoothly to $g_\infty$ on each $W_j$
with
\be\label{diam-3}
\diam_{M_i}(W_j) \le D_0 \qquad \forall i\ge j, 
\ee
\be \label{area-3}
\vol_{g_i}(\partial W_j) \le A_0,
\ee
and
\be \label{not-vol-3}
\vol_{g_i}(M\setminus W_j) \le V_j \textrm{ where } \lim_{j\to\infty}V_j=0,
\ee
Then
\be
\lim_{j\to \infty} d_{\mathcal{F}}(M_j', N')=0.
\ee
where $N'$ is the settled completion
of $(M\setminus S, g_\infty)$.
\end{thm}

\begin{rmrk}\label{necessity-well-embeddedness}
	Example \ref{ex-slit-torus} demonstrates the necessity of well-embeddedness condition in Theorem \ref{flat-to-settled}. 
\end{rmrk}

\begin{lem} \label{lem-vol-vol}
Let $M_i=(M,g_i)$ be a sequence of Riemannian manifolds
such that
there is a closed subset, $S$, and a connected precompact exhaustion,
$W_j$, of $M\setminus S$ satisfying (\ref{defn-precompact-exhaustion})
such that $g_i$ converge smoothly to $g_\infty$ on each $W_j$.
If   $\vol_{g_\infty}(M \setminus S) < \infty$ and 
\be \label{lem-edge-volume}
\vol_{g_i}(M\setminus W_j) \le V_j \textrm{ where } \lim_{j\to\infty}V_j=0,
\ee

then there exists a uniform $V_0>0$ such that
\be \label{lem-vol}
\vol_{g_i}(M) < V_0.
\ee
\end{lem}

\begin{prop} \label{prop-improve-sw}
Suppose we have a sequence of manifolds, $M_j=(M, g_j)$
with a uniform lower bound on Ricci curvature and
\be \label{prop-vol-sw}
\vol(M_j) \le V_0,
\ee
converging smoothly
away from a singular set to $(M\setminus S, g_\infty)$.
Suppose also that $(M, g_j)$ converge in the
intrinsic flat sense to $N'$ where $N'$ is the settled
completion of $(M\setminus S, g_\infty)$.  Then
\be
d_{GH}(\bar{M}_j, \bar{N}) \to 0,
\ee
and $\bar{N}=N'$.
\end{prop}

\begin{lem} \label{lem-Ricci-diam}
Suppose we have a sequence of manifolds, $M_j=(M, g_j)$
with nonnegative Ricci curvature and
\be \label{prop-vol-sw}
\vol(M_j) \le V_0,
\ee
converging smoothly
away from a singular set to $(M\setminus S, g_\infty)$
then
\be\label{diam-2}
\diam_{M_i}(W_j) \le \diam(M_i) \le D_0 \qquad \forall i\ge j.
\ee
\end{lem}

\begin{thm}\label{c-smooth-to-GH}
Let $M_i=(M,g_i)$ be a sequence of compact oriented Riemannian manifolds
with a uniform linear contractibility
function, $\rho$, which converges smoothly away from
a singular set, $S$.  If there is a uniformly well embedded
connected precompact exhaustion of
$M\setminus S$  as in (\ref{defn-precompact-exhaustion})
satisfying the volume conditions (\ref{area-4}) and (\ref{not-vol-4})
then
\be
\lim_{j\to \infty} d_{GH}(M_j, N)=0,
\ee
where $N$ is the settled and
metric completion of $(M\setminus S, g_\infty)$.
\end{thm}

\begin{proof}
	By Lemma~\ref{lem-vol-vol}, we have
\be \label{vol-4}
\vol(M_i) \le V_0.
\ee
This combined with the uniform contractibility function
allows us to apply the Greene-Petersen Compactness Theorem.
In particular we have a uniform upper bound on diameter:
\be\label{diam-4}
\diam(M_i) \le D_0,
\ee
We may now apply Theorem~\ref{flat-to-settled} to obtain
\be
    \lim_{j\to \infty} d_\mathcal{F}(M_j, N') =0
\ee
We then apply Theorem~\ref{thm-sw-contractible}
to see that the flat limit and Gromov-Hausdorff
limits agree due to the existence of the uniform
linear contractibility function
and the fact that the volume
is bounded below uniformly by the smooth limit.
In particular the metric completion and the
settled completion agree.

\end{proof}


\section{Hausdorff Measure Estimates $\Longrightarrow$ Well-Embeddedness.}\label{main-result-I}
We now prove Theorems \ref{codim-thm}, \ref{Ricci-thm-improved} and its counterpart (with Ricci condition replaced by contractibility condition) stated in the introduction. First we must prove the following two lemmas:

\begin{lem}\label{codim-lem}
 Let $M^n$ be compact Riemannian manifold, $S$ a subset of $M$ with $H^{n-1}(S) = 0 $, and let $\gamma: [0,L] \to M$ be a shortest geodesic parametrized by arclength with endpoints $x , y \in  M \setminus S$. Then, for any small enough $\epsilon>0$, there exists a path $\gamma_\epsilon$ joining $x , y$ such that $\gamma_\epsilon \cap S = \emptyset$ and 
 \be
 	L(\gamma_\epsilon) \le L(\gamma)+\epsilon.
 \ee
\end{lem}

\begin{proof}
Let $\Gamma : [-\sigma , \sigma]^{n-1} \times [0,L] \subset \R^n \to M $ be the $(n-1)-$th variation of $\gamma$  given by
\begin{eqnarray}
	\Gamma(t_1, t_2 , .... t_{n-1} , s) &=& \operatorname{exp}_{\gamma(s)} \mathbf{F}(t_1, t_2 , .... t_{n-1} , s)  
\end{eqnarray}
where
\be
	\mathbf{F}(t_1, t_2 , .... t_{n-1} , s) = \sin \left(\frac{\pi s}{L} \right) \Big( \sum_i t_i e_i (s) \Big),
\ee
in which $\{e_i(s)\}$ is a parallel orthonormal frame along $\gamma$ and $e_0(s)$ is the unit tangent to $\gamma$. 

For any $\bar{t}=(t_1,..t_{n-1})$, the curve $\gamma_{\bar{t}}(s) := \Gamma(t_1, t_2 , .... t_{n-1} , s)$ is a curve from $x$ to $y$. If we choose $\sigma$ sufficiently small then,
\be
	L(\gamma) \le L(\gamma_{\bar{t}}) \le L (\gamma) + \epsilon,
\ee
therefore, to prove the lemma, we need to find a $\bar{t}$ such that 
\be
	\gamma_{\bar{t}} \cap S= \emptyset.
\ee	

\noindent \textbf{Claim:} There exists some $\sigma>0$ such that after restricting the domain of $\Gamma$ accordingly, for any small $ \delta > 0$, $\Gamma$ is bi-Lipschitz on 
\be
	\Lambda_\delta =  [-\sigma , \sigma]^{n-1} \times [\delta , L - \delta].
\ee

To see this, we need to compute the derivative of $\Gamma$. Let
\be
	x(u,s,t_1, \dots , t_{n-1}) = \operatorname{exp}_{\gamma(s)} \Big(u \mathbf{F}(s,t_1 , \dots , t_{n-1})  \Big),
\ee
for fixed $s , t_1 , \dots , t_{n-1}$, as $u$ ranges from $0$ to $1$, the curve $x(u,s,t_1, \dots , t_{n-1})$ is a geodesic segment from $\gamma(s)$ to $\Gamma(s , t_1 , \dots , t_{n-1})$. As $s$ varies, $x$ is a variation through geodesics therefore, 
\begin{eqnarray}
	D (\Gamma)  \left( \frac{\partial}{\partial s} \right) (s, t_1 , \dots , t_{n-1})= \mathbf{J}(1),
\end{eqnarray}
where, $\mathbf{J}$ is the Jacobi field along this geodesic segment, with the initial conditions $\mathbf{J}(0) = e_0(s)$ , $\nabla_{\frac{\partial x}{\partial u}} \mathbf{J}(0) = \frac{\pi}{L} \cos \left(\frac{\pi s}{L} \right)  \sum_i t_i e_i (s) $ since,

\begin{eqnarray}
	\nabla_{\frac{\partial x}{\partial u}} \mathbf{J}(0) &=& \nabla_{\frac{\partial x}{\partial u}} \frac{\partial x}{\partial s}(0,s,t_1, \dots , t_{n-1}) \notag \\ &=& \nabla_{\frac{\partial x}{\partial s}} \frac{\partial x}{\partial u}(0,s,t_1, \dots , t_{n-1}) \\ &=& \nabla_{\frac{\partial x}{\partial s}} \mathbf{F}(0,s,t_1 , \dots , t_{n-1}) \notag \\ &=& \frac{\pi}{L} \cos \left(\frac{\pi s}{L} \right) \sum_i t_i e_i (s) . \notag
\end{eqnarray}

Also for any $1 \le i \le n-1$ we have:
\begin{eqnarray}
	D (\Gamma)  \left( \frac{\partial}{\partial t_i} \right) (s, t_1 , \dots , t_{n-1}) = \left( \operatorname{exp}_{\gamma(s)} \right)_*  \big|_{\mathbf{F}(s,t_1 , \dots , t_{n-1})} \sin \left(\frac{\pi s}{L} \right) e_i(s).
\end{eqnarray}

For $t_1 = \dots , t_{n-1} = 0$, and $V = \alpha_0 \frac{\partial}{\partial s} + \sum_i  \alpha_i \frac{\partial}{\partial t_i}$ with $||V|| =1 $ we compute:
\begin{eqnarray}
	||D(\Gamma) V || &=& \left\|D(\Gamma) \Big( \alpha_0 \frac{\partial}{\partial s} + \sum_i  \alpha_i \frac{\partial}{\partial t_i} \Big) \right\|  \notag  \\ &=&  \left\| \alpha_0 e_0(s) + \sin \left(\frac{\pi s}{L} \right) \sum_i  \alpha_i e_i(s) \right\|,
\end{eqnarray}

Therefore, for any $\delta > 0$ , there exist $c(\delta)> 0$ such that on $\gamma$,
\be
	 0 < c(\delta) \le || D(\Gamma)V || \le 1.
\ee

By continuity, for a small enough $\sigma$, we will have
\be
	 0 < \frac{c(\delta)}{2} \le || D(\Gamma)V || \le 3/2,
\ee
on $[-\sigma , \sigma]^{n-1} \times [\delta , L -\delta]$. 

Also by making $\sigma$ smaller, we can assume that 
\be
	\sigma < \frac{ r_{focal}}{n},
\ee
in which, $r_{focal}$ is the focal radius of the geodesic $\gamma$. This guarantees that $\Gamma$ is injective on $[-\sigma , \sigma]^{n-1} \times [\delta , L - \delta]$ which is compact, therefore $\Gamma$ is a homeomorphism with the derivative bounded away from zero on its domain. By applying the inverse function theorem, we deduce that $\Gamma$ is a diffeomorphism on $[-\sigma , \sigma]^{n-1} \times [\delta , L - \delta]$ onto its image. As a result, $\Gamma$ is bi-Lipschitz on $\Lambda_{\delta}$. 

It is rather straightforward to see that for a Lipschitz function $f : \R^n \to \R^m$, any subset $A \subset \R^n$ and $0 \le s < \infty$, we have
\be\label{lip-hausdorff}
	H^s \left( f(A) \right) \le [\Lip f]^s H^s(A)
\ee
(see \cite[Theorem 3.1.2]{Lin-Yang-GMT-text}). Therefore, bi-Lipschitz preimages of sets of $0$ Hausdorff measure, have $0$ Hausdorff measure. Since $\Gamma$ is bi-Lipschitz on $\Lambda_{\delta}$, we get:
\be\label{bi-lip-1}
	H^{n-1} \left( \Gamma^{-1}(S) \cap \Lambda_\delta \right) = 0.
\ee

Now we can compute:
\begin{eqnarray}
	H^{n-1} \left(  \Gamma^{-1}(S) \right) &\le& H^{n-1} \left(  \bigcup_i \left(\Gamma^{-1}(S) \cap \Lambda_{1/i} \right) \right) \notag \\ &&+ H^{n-1} \left(  \Gamma^{-1}(S) \cap [-\sigma , \sigma]^{n-1} \times \{0\}   \right)\notag \\ &&+ H^{n-1} \left(  \Gamma^{-1}(S) \cap [-\sigma , \sigma]^{n-1} \times \{L\}   \right)\notag \\ && \;=0. \notag
\end{eqnarray}

Since any orthogonal projection $\operatorname{Pr}: \R^n \to \R^k$ is distance dicreasing, we have $\Lip \left(\operatorname{Pr} \right) \le 1$. By (\ref{lip-hausdorff}) (see \cite[Theorem 3.1.2]{Lin-Yang-GMT-text}), for any $A \subset \R^n$ and any $0 \le s < \infty$, we get
\be
	H^s \left( \operatorname{Pr}(A) \right) \le H^s(A)
\ee 

Thus for any orthogonal projection $\operatorname{Pr}$ onto an $(n-1)-$dimensional face of $[-\sigma , \sigma]^{n-1} \times [0,L]$ we have
\be
	H^{n-1} \left(\operatorname{Pr}\left( \Gamma^{-1}(S) \right) \right) = 0.
\ee

Let $\operatorname{Pr_1}$ and $\operatorname{Pr_2}$ be the projections onto the faces $[- \sigma , \sigma]^{n-1} \times \{0\}$ and $[- \sigma , \sigma]^{n-1} \times \{L\}$ respectively. Setting $E_1 = \operatorname{Pr_1}\left(\Gamma^{-1}(S) \right)$ and $E_2 = \operatorname{Pr_2}\left(\Gamma^{-1}(S) \right) $, then,
\be
	H^{n-1} ( E_i) = 0  \;\; for\;\; i=1,2
\ee
and so,
\be
	H^{n-1} ([-\sigma , \sigma]^{n-1} \setminus E_i) = (2 \sigma)^{n-1} \;\;for\;\; i=1,2
\ee	

Any countable union of null sets is a null set (see \cite[p.~ 26]{Folland}) hence, $E_1 \cup E_2$ is a null set. This means that
\be
	H^{n-1} \left( [-\sigma , \sigma]^{n-1} \setminus E_1 \cup [-\sigma , \sigma]^{n-1} \setminus E_2 \right) = (2 \sigma)^{n-1}.
\ee

Let $ \bar{t} = (t_1 , t_2 , ...., t_{n-1}) \in (([-\sigma , \sigma]^{n-1} \setminus E_0)   \cap ([-\sigma , \sigma]^{n-1} \setminus E_L)$, then the path 
\be
	\gamma_{\bar{t}} (s) = \Gamma (t_1 , t_2 , t_3 ,... , s ),
\ee
is a path joining $x,y$ and $\gamma_{\bar{t}} \cap S = \emptyset$ which also satisfies
\be
	| L (\gamma_\epsilon) - L(\gamma)  | \le \epsilon.
\ee
\end{proof}

\begin{lem}\label{codim-lambda-lem}
Let $M^n$ be a compact Riemannian manifold, $S$ a set with $H^{n-1}(S) = 0$ and $\diam_{g_\infty}(M\setminus S)<\infty$ then, 
any connected precompact exhaustion, $W_j$, 
of $M^n \setminus S$ is uniformly well embedded.
\end{lem}

\begin{proof}
We claim for fixed $i,j$,
\be
\lambda_{i,j}=\limsup_{k\to \infty} \lambda_{i,j,k} =0.
\ee

Suppose not. Let $x_{i,j,k}, y_{i,j,k}\subset \bar{W}_j$ achieve to
supremum in the definition of $\lambda_{i,j}$.

Since $\bar{W_j}$ is compact, a subsequence
as $k\to \infty$ converges to $x_{i,j}, y_{i,j} \subset \bar{W}_j$.
Let $\gamma_{i,j}$ be a minimizing geodesic between
these points in $M$ with respect to $g_i$.  Since $S$ is
a set of codimension strictly larger than than $1$, by applying Lemma \ref{codim-lem}, we can find a curve
$C_{i,j}: [0,1]\to M\setminus S$ between these points such that
\be
L_{g_i}(C_{i,j}) \le d_{M, g_i}(x_{i,j}, y_{i,j}) + \lambda_{i,j}/5,
\ee

Let $k$ be chosen from the subsequence sufficiently large that
\begin{eqnarray*}
& C_{i,j}([0,1]) \subset W_k,\\
&d_{(\bar{W}_j,g_i)}(x_{i,j,k}, x_{i,j}) < \lambda_{i,j}/10,\\
&d_{(\bar{W}_j,g_i)}(y_{i,j,k}, y_{i,j}) < \lambda_{i,j}/10,
\end{eqnarray*}
Thus
\begin{eqnarray*}
d_{(\bar{W}_k,g_i)}(x_{i,j,k}, y_{i,j,k})
&\le&
d_{(\bar{W}_k,g_i)}(x_{i,j,k}, x_{i,j})+
d_{(\bar{W}_k,g_i)}(x_{i,j}, y_{i,j}) +
d_{(\bar{W}_k,g_i)}(y_{i,j}, y_{i,j,k}) \\
&\le&
d_{(\bar{W}_j,g_i)}(x_{i,j,k}, x_{i,j})+
L(C_{i,j}) +
d_{(\bar{W}_j,g_i)}(y_{i,j}, y_{i,j,k}) \\
&\le&
\lambda_{i,j}/10+
d_{M, g_i}(x_{i,j}, y_{i,j}) + \lambda_{i,j}/5
+\lambda_{i,j}/10 \\
&\le&
2\lambda_{i,j}/5+
d_{M, g_i}(x_{i,j}, x_{i,j,k}) +
d_{M, g_i}(x_{i,j,k}, y_{i,j,k}) +
d_{M, g_i}(y_{i,j,k}, y_{i,j})
 \\
 &\le&
2\lambda_{i,j}/5+
d_{W_j, g_i}(x_{i,j}, x_{i,j,k}) +
d_{M, g_i}(x_{i,j,k}, y_{i,j,k}) +
d_{W_j, g_i}(y_{i,j,k}, y_{i,j})
 \\
 &\le&
3\lambda_{i,j}/5+
d_{M, g_i}(x_{i,j,k}, y_{i,j,k})
 \\
 &\le&
3\lambda_{i,j}/5+
d_{W_k, g_i}(x_{i,j,k}, y_{i,j,k}) -\lambda_{i,j,k},
 \end{eqnarray*}
 by the choice of $x_{i,j,k}$ and $y_{i,j,k}$.  This is
 a contradiction.

 Next we must show
 \be
\limsup_{j\to \infty} \limsup_{k\to \infty} \limsup_{i\to\infty} \lambda_{i,j,k}
\le \lambda_0.
\ee
  Observe that
 \be
 \lambda_{i,j,k} \le \bar{\lambda}_{i,j,k}=\diam_{(W_{k}, g_i)}(W_j).
 \ee
 Since $g_i \to g_\infty$ on $W_k$ we know
 \be
 \limsup_{i\to \infty} \lambda_{i,j,k}
 \le \diam_{(W_{k}, g_\infty)}(W_j).
 \ee
 
\textbf{Claim:} 
 \be\label{claim-for-well-embedded}
 	\limsup_{k\to\infty} \diam_{\left(W_k , g_\infty \right)} (W_j) \le \diam_{\left(M \setminus S , g_\infty \right)} ( W_j ).
 \ee
 
 Suppose not; then, there exists $s > 0$ and a subsequence $k \to  \infty$ such that
 \be
 	\limsup_{k\to\infty} \diam_{\left(W_k , g_\infty \right)} (W_j) = L > \diam_{\left(M \setminus S , g_\infty \right)} ( W_j ) + 5\delta.
 \ee
 
 Pick $x_k , y_k \in W_j$ so that
 \be
 	\diam_{\left(W_k , g_\infty \right)} (W_j) \le d_{\left(W_k , g_\infty \right)} (x_k , y_k) + \delta.
 \ee
 
$\bar{W}_j$ is compact therefore, after passing to a subsequence,
\begin{eqnarray}
	x_k \to x \in \bar{W}_j\\
	y_k \to y \in \bar{W}_j.
\end{eqnarray}
 
Therefore, there exists a curve $c : [0,1] \to M \setminus S $ such that,
\begin{eqnarray}
	L_{g_\infty}(c) &<& d_{(M\setminus S , g_\infty)}(x,y) + \delta \\ &<& \diam_{(M \setminus S , g_\infty)}(\bar{W}_j) + \delta \\ &\le& \diam_{(M \setminus S , g_\infty)}(W_j) + \delta \\ &<& L - 4\delta.
\end{eqnarray}
 
For $k$ sufficiently large, we have 
\be
	c([0,1]) \subset W_k,
\ee
so 
\begin{eqnarray}
	L - 4\delta > L_{g_\infty} (c) &>& d_{\left(W_k , g_\infty \right)} (x,y) \\ &>& d_{\left(W_k , g_\infty \right)} (x_k,y_k) - 2\delta \\ &\ge& \diam_{\left(W_k , g_\infty \right)} (W_j) - 3\delta.
\end{eqnarray}
 
Taking the limit as $k \to \infty$, we get:
\be
	L - 4\delta \ge L.
\ee
which is a contradiction hence, the claim is proved and we have
 \be
 \limsup_{k\to\infty}\limsup_{i\to \infty} \lambda_{i,j,k}
 \le \diam_{(M \setminus S, g_\infty)}(W_j),
 \ee
 and so
 \be
\limsup_{j\to \infty} \limsup_{k\to \infty} \limsup_{i\to\infty} \lambda_{i,j,k}
\le \diam_{g_\infty}(M\setminus S).
\ee
\end{proof}

\textbf{Proof of Theorem \ref{codim-thm}:}
\begin{proof}
	The lemmas \ref{codim-lem} and \ref{codim-lambda-lem} prove the well-embeddedness and then applying the Theorem \ref{flat-to-settled} completes the proof of Theorem \ref{codim-thm}.
\end{proof}

\begin{rmrk} \label{codim-necessity}
Example \ref{ex-pinched-torus} demonstrates that the connectivity of the exhaustion in Theorems~\ref{codim-thm} hence, in Theorems \ref{Ricci-thm-improved} and \ref{c-codim-thm} is a necessary condition. The excess volume bound in (\ref{not-vol-2}) is shown to
be necessary in ~\cite[Example 3.7]{Lakzian-Sormani}.
All these examples satisfy the uniform embeddedness hypothesis
of Theorem~\ref{flat-to-settled} and demonstrate the necessity of these
conditions in that theorem as well.   
By Lemma~\ref{lem-Ricci-diam}, the diameter hypothesis
is not necessary when the Ricci curvature is nonnegative
although the volume condition is still necessary as seen in
~ \cite[Example 3.8]{Lakzian-Sormani}.
Otherwise we see this is a necessary condition in Example \ref{ex-neg-Ric-bound}.
We were unable to find an example proving the necessity of
the uniform bound on the boundary volumes, (\ref{area-2}), 
and suggest this as an open question in ~\cite[Remark 3.15]{Lakzian-Sormani}.
The Hausdorff measure condition $H^{n-1}(S)=0$ of Theorem~\ref{codim-thm}
and the 
uniform embeddedness hypothesis of Theorem~\ref{flat-to-settled}
are seen to be necessary for their respective theorems
in \ref{ex-slit-torus}.
\end{rmrk}

\textbf{Proof of Theorem \ref{Ricci-thm-improved}:}

\begin{proof}  
	The assumption $H^{n-1}(S)=0$ along with the hypotheses (\ref{diam-2}), (\ref{area-2}) and (\ref{not-vol-2}), allows us to apply Theorem \ref{codim-thm}. Therefore, $(M_i , g_i)$ has an intrinsic flat limit and this limit coincides with the settled completion  of $(M \setminus S , g_\infty)$. Now by proposition \ref{prop-improve-sw}, the Gromov-Hausdorff and Intrinsic Flat limits agree.
\end{proof}

\begin{rmrk}\label{necessity-diam-codim}
	Example \ref{ex-neg-Ric-bound} proves the necessity of the the condition (\ref{diam-2}) in Theorem \ref{Ricci-thm-improved}.
\end{rmrk}

\begin{rmrk}\label{not-vol-necessity}
From ~\cite{ChCo-PartII},  Ricci bounded below and $d_{GH}(M_i , M) \to 0$ imply $\vol(M_i) \to \vol(M)$ (conjectured by Anderson-Cheeger) and $M_i$ are homeomorphic to $M$ for $i$ sufficiently large  (later proved diffeomorphic in ~\cite{Perelman-max-vol} .) This means that (\ref{not-vol-2}) in Theorem \ref{Ricci-thm-improved} is a necessary condition. 
\end{rmrk}

\begin{thm}\label{c-codim-thm}
Let $M_i=(M,g_i)$ be a sequence of oriented compact Riemannian manifolds
with a uniform linear contractibility
function, $\rho$, which converges smoothly away from
a  closed singular subset, $S$, with $H^{n-1}(S) = 0$.  
If there is a connected precompact exhaustion of
$M\setminus S$  as in (\ref{defn-precompact-exhaustion})
satisfying the volume conditions
\be \label{area-4}
\vol_{g_i}(\partial W_j) \le A_0
\ee
and
\be \label{not-vol-4}
\vol_{g_i}(M\setminus W_j) \le V_j \textrm{ where } \lim_{j\to\infty}V_j=0,
\ee
then
\be
\lim_{j\to \infty} d_{GH}(M_j, N)=0,
\ee
where $N$ is the settled and
metric completion of $(M\setminus S, g_\infty)$.
\end{thm}

\begin{proof}
	By the proof of Theorem \ref{c-smooth-to-GH}, we see that
\be
	\diam(M_i) \le D_0,
\ee	
This along with $H^{n-1}(S)=0$, (\ref{area-4}) and (\ref{not-vol-4}), allows us to apply the Lemma \ref{codim-lambda-lem} to get the well-embeddedness of the exhaustion $\{W_j\}$.  Then, we can fully apply Theorem \ref{c-smooth-to-GH} and that finishes the proof.
\end{proof}


\section{Diameter Controls $\Longrightarrow$ Well-embeddedness.}\label{main-result-II}
In this section we prove Theorems \ref{diam-thm}, \ref{Ricci-diam-thm} and its counterpart (with Ricci condition replaced with contractibility condition.) In the theorems of this section, there is no co-dimension condition on the singular set $S$. We first need to prove the following lemma:

\begin{lem}\label{lem-diameter-control-1}
Suppose $W_j$ is a connected precompact exhaustion of $M \setminus S$ with boundaries $\partial W_j$
such that any connected component of $M \setminus W_j$ has a connected boundary.
If the intrinsic diameters satisfy
\be
\diam_{(W_j, g_i)}(W_j)\le D_{int},
\ee
and
\be \label{new-bridge-1}
\limsup_{i \to \infty} \;  \left\{ \sum_\beta \diam_{(\Omega^\beta_j , g_i)} (\Omega^\beta_j) :  \text{$\Omega^\beta_j$ connected component of $\partial W_j$} \right\} \;
\le B_j
\ee
satisfies $\lim_{j \to \infty} B_j = 0$
then $W_j$ is uniformly well embedded.
\end{lem}

\begin{proof}
Recall from Definition \ref{well-embedded} that we have
\be
\lambda_{i,j,k}= \sup_{x,y\in W_j} |d_{(W_{k}, g_i)}(x,y)- d_{(M,g_i)}(x,y)| \le \diam_{(W_k , g_i)}(W_j).
\ee
Since 
\be
\diam_{W_k, g_i}(W_j)\le \diam_{W_k, g_i}(W_k),
\ee
and
$g_i$ converges smoothly on $W_k$ we have,
\be
    \lim_{i \to \infty} \diam_{(W_k , g_i)}(W_k) = \diam_{(W_k , g_\infty)}(W_k) \le D_{int}.
\ee

Therefore,
\be
\limsup_{j\to \infty} \limsup_{k\to \infty} \limsup_{i\to\infty} \lambda_{i,j,k}
\le \limsup_{j\to \infty} \limsup_{k\to \infty} \lim_{i\to\infty} \diam_{(W_k , g_i)}(W_k) \le D_{int}.
\ee

Now suppose $x_{ijk} , y_{ijk} \in \partial{W}_j$ give the supremum in the definition of $\lambda_{ijk}$, Let $\gamma_{ijk}$ and $C_{ijk}$ be shortest paths between $x_{ijk}$ and $y_{ijk}$ in $\bar{W_k}$ and $M$ respectively.  Letting $k \to \infty$ and passing to a subsequence if necessary,  $x_{ijk} \to x_{ij}  \in \bar{W_j}$ and $y_{ijk} \to y_{ij}  \in \bar{W_j}$.  Passing to a subsequence again if necessary, $C_{ijk}$ converges to $C_{ij}$ which is a shortest path between $x_{ij}$ and $y_{ij}$ in $M$ (c.f. \cite{BBI}[Prop 2.5.17]). And let $\gamma^k_{ij}$ be the shortest path between $x_{ij}$ and $y_{ij}$ in $\bar{W_k}$. 

We will estimate $C_{ij}$ by curves in $\partial W_{j}$ with controlled increase in the length. Denote the curve obtained in $n$th step by $C^n_{ij}$ and let $C^0_{ij} = C_{ij}$. To obtain $C^{n+1}_{ij}$ from $C^{n}_{ij}$, we proceed as follows: Suppose $\{\Omega_j^n\}_{n\in {N}}$ are the connected components met by $C_{ij}$ (in more than one point). If $C^n_{ij}$ does not intersect $\Omega_j^{n+1}$, then we let $C^{n+1}_{ij} = C^{n}_{ij}$ . If $C^n_{ij}$ intersects $\Omega_j^{n+1}$ then define
\be
	t_1 = \inf \left\{ t : C^n_{ij}(t) \in \Omega_j^{n+1} \right\},
\ee
and
\be
	t_2 = \sup \left\{ t : C^n_{ij}(t) \in \Omega_j^{n+1} \right\}.
\ee 
  
Since $\Omega_j^{n+1}$ is connected, we can replace the segment $C^n_{ij}[t_1 , t_2]$ with a shortest path in $\Omega_j^{n+1}$. The curve obtained in this way is our $C^{n+1}_{ij}$. Note that connectivity of the boundary components of $M \setminus W_j$ implies that if $C^n_{ij}$ enters $M \setminus W_j$ through $\Omega_j^{n+1}$ at time $t$, then it has to  intersect $\Omega_j^{n+1}$ again at time $t' > t$ in order to enter $ W_j $. 

This construction implies that for all $n$,
\be
	L(C^n_{ij}) \le L(C^0_{ij}) + B_j,
\ee
hence, the sequence $\{C^n_{ij}\}_{n\in {N}}$ obtained in this way have uniform bounded length and as a result, we can apply the Arzela-Ascoli's theorem to obtain, after possibly passing to a subsequence, a  limit $C'_{ij}$ i.e. $C'_{ij}$ is a curve with end points $x_{ij},y_{ij}$ and there are parametrizations of $\{C^n_{ij}\}_{n\in {N}}$ and $C'_{ij}$ on the same domain such that $\{C^n_{ij}\}_{n\in {N}}$ uniformly converges to $C'_{ij}$.  We claim that $C'_{ij}$ is contained in $\bar{W_j}$.  For any $t$, tracing the curve $C^0_{ij}$ back and forth from the point $C^0_{ij}(t)$, we reach two immediate components $\Omega_j^l$ and $\Omega_j^m$ and this means that for $n \ge \max\{l , m\}$, $C^n_{ij}(t) \in \bar{W_j}$ and since $C^n_{ij}(t) \to C'_{ij}(t)$ we must have $C'_{ij}(t) \in \bar{W_j}$ . Furthermore, for $i$ large enough (depending on $j$)
\begin{eqnarray}\label{eq:lemdiam1}
 	\limsup_{i \to \infty} \left( L(\gamma^k_{ij}) - L(C_{ij}) \right) &\le& \limsup_{i \to \infty}  \left( L(C'_{ij}) - L(C_{ij}) \right) \notag \\ &\le& \limsup_{i \to \infty} \sum_\beta \diam^{g_i}_{\Omega^\beta_j}\left(\Omega^\beta_j\right)\\ &< & B_{j}. \notag
\end{eqnarray}

Also
\begin{eqnarray}\label{eq:lemdiam2}
          \lim_{k \to \infty} \lambda_{ijk} &=& \lim_{k \to \infty}  \left(  d_{(W_k , g_i)} \left( x_{ijk} , y_{ijk} \right) - d_{(M , g_i)} \left( x_{ijk} , y_{ijk} \right)  \right) \notag \\
          &\le&  \lim_{k \to \infty}  \left(  d_{(W_m , g_i)} \left( x_{ijk} , y_{ijk} \right) - d_{(M , g_i)} \left( x_{ijk} , y_{ijk} \right)  \right) \\ & =&  L(\gamma^m_{ij}) - L(C_{ij}). \notag
\end{eqnarray}

Therefore, combining (\ref{eq:lemdiam1}) and (\ref{eq:lemdiam2}) we have
\be
	\lambda_{j} = \limsup_{i \to \infty} \lim_{k \to \infty} \lambda_{ijk}  \le B_{j},
\ee
hence,
\be
	\limsup_{j \to \infty} \lambda_j  \le  \lim_{j \to \infty} B_{j}  = 0.
\ee
\end{proof}

\textbf{Proof of Theorem \ref{diam-thm}:}
\begin{proof}
The Lemma \ref{lem-diameter-control-1} combined with Theorem~\ref{flat-to-settled}
completes the proof of Theorem~\ref{diam-thm}.  Recall
Definition~\ref{well-embedded}.
\end{proof}

\begin{rmrk}\label{necessity-diam}
 Example \ref{ex-pinched-torus} demonstrates that the connectivity of the exhaustion in Theorem \ref{diam-thm}, hence in Theorems \ref{Ricci-diam-thm} and \ref{c-diam-GH} is a necessary condition. Example \ref{ex-pulled-torus} demonstrates that in Theorem \ref{diam-thm}, hence, in Theorems \ref{Ricci-diam-thm}  and  \ref{c-diam-GH}, the condition on the intrinsic diameter can not be replaced by the same condition on the extrinsic diameter. The necessity of other conditions  follow as in Remark \ref{codim-necessity}.
\end{rmrk}

\textbf{Proof of Theorem~\ref{Ricci-diam-thm}:}
\begin{proof}
The hypothesis $\diam(M_j) \le D_0$ combined with the hypothesis including (\ref{area-2}) and (\ref{not-vol-2}),
allows us to apply Theorem~\ref{diam-thm}.   So 
$(M_i, g_i)$ has an intrinsic flat limit and this intrinsic
flat limit is the settled completion of $(M\setminus S, g_\infty)$. Thus by
Proposition~\ref{prop-improve-sw}, the Gromov-Hausdorff
and Intrinsic Flat limits agree. 
\end{proof}

\begin{rmrk}\label{necessity-diam-diam}
	Example \ref{ex-neg-Ric-bound} proves the necessity of the the condition (\ref{diam-M}) in Theorem \ref{Ricci-diam-thm}.
\end{rmrk}

\begin{rmrk}\label{not-vol-necessity}
From ~\cite{ChCo-PartII},  Ricci bounded below and $d_{GH}(M_i , M) \to 0$ imply $\vol(M_i) \to \vol(M)$ (conjectured by Anderson-Cheeger) and $M_i$ are homeomorphic to $M$ for $i$ sufficiently large  (later proved diffeomorphic in ~\cite{Perelman-max-vol} .) This means that (\ref{not-vol-2}) in Theorem \ref{Ricci-diam-thm} is a necessary condition. 
\end{rmrk}

\begin{thm}\label{c-diam-GH}
Let $M_i=(M,g_i)$ be a sequence of Riemannian manifolds
with a uniform linear contractibility
function, $\rho$,
which converges smoothly away from
a closed singular set, $S$, with
\be 
\vol(M_i) \le V_0.
\ee  

If there is a connected precompact exhaustion, $W_j$, of
$M\setminus S$, satisfying (\ref{defn-precompact-exhaustion})
such that 
each connected component of $M\setminus W_j$ has a connected boundary, satisfying (\ref{new-bridge-2})-(\ref{m-int-diam}),
(\ref{m-area}) and (\ref{m-edge-volume})
then
\be
\lim_{j\to \infty} d_{GH}(M_j, N)=0.
\ee
where $N$ is the metric completion
of $(M\setminus S, g_\infty)$.
\end{thm}

\begin{proof}
By Lemma~\ref{lem-vol-vol}, we have
\be \label{vol-4}
\vol(M_i) \le V_0.
\ee

This combined with the uniform contractibility function
allows us to apply the Greene-Petersen Compactness Theorem.
In particular we have a uniform upper bound on diameter
\be\label{diam-4}
\diam(M_i) \le D_0,
\ee

We may now apply Theorem~\ref{diam-thm}
to obtain
\be
    \lim_{j\to \infty} d_\mathcal{F}(M_j, N') =0.
\ee

We then apply Theorem~\ref{thm-sw-contractible}
to see that the flat limit and Gromov-Hausdorff
limits agree due to the existence of the uniform
linear contractibility function
and the fact that the volume
is bounded below uniformly by the smooth limit.
In particular the metric completion and the
settled completion agree.
\end{proof}


\section{Examples} \label{Sect-examples}

In this section we present some examples which helps in understanding the notions we have mentioned so far. Some examples will prove the necessity of some conditions in Theorem \ref{Ricci-diam-thm}.

\subsection{Unbounded Limits}
The following examples show why some sort of bounded geometry is necessary in this context.

\begin{example}\label{ex-not-compact}
There are metrics $g_j$ on the sphere $M^3$ 
with a uniform upper bound on volume such that
$(M^3, g_j)$ converge smoothly away from a point singularity
$S=\{p_0\}$ to a complete noncompact manifold.  There is
no Gromov-Hausdorff limit in this case.   The intrinsic flat
limit is $(M\setminus S, g_\infty)$. 
\end{example}

\begin{proof}
Let
\be
g_0= h^2(r) dr^2 + f^2(r) g_{S^2},
\ee
be defined on $M^3\setminus S$ as a complete metric such that
\be
\int_0^\pi h(r) dr =\infty,
\ee
and
\be \label{ex-vol-eqn}
\int_0^\pi \omega_2 h(r) f^2(r) dr<\infty,
\ee
so that $\diam (M\setminus S, g_0)=\infty$ and $\vol(M\setminus S, g_0)<\infty$.

We set
\be
g_j= h_j^2(r) dr^2 + f_j^2(r) g_{S^2},
\ee
such that
\begin{eqnarray}
h_j(r)&=&h(r) \qquad r\in [0, \pi-1/j],\\
f_j(r)&=&f(r) \qquad r\in [0, \pi-1/j],
\end{eqnarray}
and extend smoothly so that $g_j$ is a metric on $S^3$.

Metrics $g_j$ converge smoothly to $g_0$ away from $S=\{p_0\}=r^{-1}(\pi)$
and, since $(M\setminus S, g_0)$ is noncompact, $(M, g_j)$
has no Gromov-Hausdorff limit.   The intrinsic flat
limit of $(M, g_j)$ is the settled completion of $(M\setminus S, g_0)$
by Theorem~\ref{codim-thm}, 
taking $W_j=r^{-1}[0, \pi-1/j)$
since
\be
\int_{\pi-1/k}^\pi \omega_2 h(r) f^2(r) dr =0,
\ee
by the finiteness of (\ref{ex-vol-eqn}) and we also have
$\vol_{g_i}(\partial(W_j))\le f^2(r)$.   In this
case the settled completion is just $(M\setminus S, g_0)$
because it is already a complete metric space with positive
density.
\end{proof}

\begin{example}\label{ex-no-F}
There are metrics $g_j$ on $M^3=S^3$ converging smoothly
away from a singular set $S=\{p_0\}$ to a complete noncompact
manifold of infinite volume.  $(M, g_j)$ have no intrinsic flat
or Gromov-Hausdorff limit since, if such a limit existed it would
have to contain the smooth limit and the smooth limit has
infinite diameter and volume.
\end{example}

\begin{proof}
We define
a metric $g_0$ on $M\setminus S$ exactly as
in Example~\ref{ex-not-compact} except that we
replace (\ref{ex-vol-eqn}) with
\be \label{ex-vol-eqn-2}
\int_0^\pi \omega_2 h(r) f^2(r) dr =\infty,
\ee
so that $\diam (M\setminus S, g_0)=\infty$ and $\vol(M\setminus S, g_0)=\infty$.

Selecting $g_j$ also as in that example, we have $(M, g_j)$
converge smoothly away from $S$ to $(M\setminus S, g_0)$.
However there is no Gromov-Hausdorff limit because the
diameter diverges to infinity \cite{Gromov-metric} and there is no intrinsic flat
limit because the volume diverges to infinity \cite{SorWen2}.
\end{proof}

One may define pointed Gromov-Hausdorff and pointed
intrinsic flat limits to deal with unboundedness.  However
even assuming boundedness, we see in  ~\cite[Example 3.11]{Lakzian-Sormani} that the Gromov-Hausdorff limit need not exist.

\subsection{Ricci Example}
This example shows that the mere uniform lower bound for Ricci curvature does not imply the existence of the Gromov-Hausdorff limit.

\begin{example}\label{ex-neg-Ric-bound}
There are metrics $g_j$ on $M^3=S^3$ with negative uniform lower bound on Ricci curvature, converging smoothly away from a singular set $S = \{ p_0 \}$ to a complete noncompact manifold of finite volume.
\end{example}

\begin{proof}
Consider the metric $g_0$ on $S^3 \setminus \{ p_0 \} = \R \times S^2$ given by
\be
	\bar{g}(t) = dt^2 + (\bar{f}(t))^2 g_{S^2},
\ee
where, $\bar{f}$ is a nonzero smooth function such that
\be
	\bar{f}(t) = \sin(t) \;\; \text{for}\;\; t \in [0 , \pi/2],
\ee
\be
	\bar{f}(t) = \exp{-t}  \;\; \text{for}\;\; t \in [\pi/2 + 1 , \infty),
\ee
with $\bar{f}''(t)<\bar{f}(t)$ elsewhere.
hence, $\bar{g}$ has Ricci curvature bounded below by 
\be
	- \Lambda =  - 2 \max \frac{\bar{f}''}{\bar{f}} > -\infty.
\ee
We can extract warped metrics $\bar{g}_j$ on $ [0 , j+1] \times S^2$ 
\be
	\bar{g}(t) = dt^2 +\bar{f}_j(t)^2 g_{S^2},
\ee
where $\bar{f}_j$ is a nonzero smooth function satisfying
\be
	\bar{f}_j(t) = \bar{f}(t)   \;\; \text{for}\;\; t \in [0 , j-1],
\ee
\be
	\bar{f}_j(t) = \exp{-j}\sin(\exp{j} (\pi + t - j - 1))   \;\; \text{for}\;\; t \in [j,j+1],
\ee
and
\be
	 - 2 \max \frac{\bar{f}_j''}{\bar{f}_j} \ge - \Lambda.
\ee
Note that in fact we are cutting off a part of $\bar{f}$ and replacing it with a less concave function which closes up like a $\sin$ function hence obtaining a metric on $S^3$, with lower bound on Ricci curvature.

It is clear that
\be
	\vol(M,g_0) \le \infty.
\ee

Let $\phi : [0,\pi] \to [0,\infty)$ be a smooth increasing function such that
\be \label{phi-1}
	\phi(r) = r    \;\; \text{for} \;\; r \in [0 , \pi/2],
\ee
with 
\be\label{phi-2}
	\lim_{r \to \pi} \phi(r) = \infty.
\ee

For $ j > 2$, let $\phi_j(r): [0 , \pi] \to [0 , L_j = j + \pi/2 + 1]$ be a smooth increasing function such that
\be \label{phi-3}
	\phi_j(r) = \phi(r)  \;\; \text{for} \;\; r \in [0 , \phi^{-1}(j + \pi/2)],
\ee
and
\be \label{phi-4}
	\phi_j (r) = j + r-\pi/2 + 1 \textrm{ for } r \textrm{ near } \pi.
\ee

 we construct metrics
\be
	g_j(r) = \phi_j^{*} \left(\bar{g}_j \right),
\ee
with Ricci bounded below by $- \Lambda$ converging smoothly away from $\{ p_0 \}$ to $\phi^{*}(\bar{g})$. Taking $W_j = r^{-1}([0 , \pi - 1/j))$ we observe that $W_j$ satisfies all the hypotheses in Theorem \ref{Ricci-diam-thm} except that $\diam_{M_i}(W_j)$ is not bounded. The Gromov-Hausdorff limit does not exist because $(M \setminus S , g_0) $ is complete noncompact. Both intrinsic flat limit and the metric completion coincide with the complete noncompact manifold $(M \setminus S , g_0) $.
\end{proof}

\begin{rmrk}\label{noncompact-nonneg-Ric}
From ~\cite{MR0417452}, we know that any complete noncompact manifold with nonnegative Ricci curvature has infinite volume, so Example~\ref{ex-neg-Ric-bound} can not be constructed with the sequence having nonnegative Ricci curvature.
 \end{rmrk}

\subsection{Pinching a torus}

\begin{example} \label{ex-pinched-torus}
There are $(M^2, g_j)$ all diffeomorphic to the torus, $S^1\times S^1$
which converge smoothly away from a singular set, $S=\{0\}\times S^1$,
to
\be
(M\setminus S, g_\infty)= \left( (0,2\pi)\times S^1, dt^2+ \sin^2(\frac{t}{2}) \;  ds^2\right).
\ee
So the metric completion and the settled completions are both
homeomorphic to
\be
M_\infty = [0,2\pi]\times S^1 / \sim,
\ee
where 
\be
(0, s_1) \sim (0, s_2) \;\; and \;\;  (2\pi, s_1) \sim (2\pi, s_2) \qquad \forall s_1, s_2 \in S^1.
\ee
However the Gromov-Hausdorff and Intrinsic Flat limits
identify these two end points.
\end{example}

\begin{proof}
Let $g_j$ on $M$ be defined by
\be
g_j= dt^2+ f_j^2(t) ds^2,
\ee
where $f_j:S^1 \to (0,1]$ are smooth with $|f_j'(t)|\le 1$
that decrease uniformly to $\sin(\frac{t}{2})$ and $f_j(t)=\sin(\frac{t}{2})$
for $t\in [1/j, 2\pi-1/j]$.
\end{proof}

\subsection{Examples of Slit Tori}



\begin{example}\label{ex-slit-torus}
Let $(M^2, g)$ be the standard flat 2 torus $S^1\times S^1$
and $S\subset M^2$ a 
vertical geodesic
segment of length $\le \pi$, then if $g_j$ are a constant sequence
of the standard flat metric, we see that $(M^2, g_j)$ converges
smoothly to itself and thus the intrinsic flat and Gromov-Hausdorff limits
are both the flat torus.  However, the metric completion of $(M\setminus S, g_\infty)$ has two copies of the slit, $S$ (with end points identified hence the limit has fundamental group = $\Z$) one found taking 
limits of Cauchy sequences from the right and the other found taking 
limits of Cauchy sequences from the left.   
This example shows necessity of uniform well embeddedness
condition in our Theorems.
\end{example}

\begin{proof}
Let $M^2= S^1\times S^1 =[0,2\pi]\times[0,2\pi]/ \sim$
such that $(x,0)\sim(x,2\pi)$ and $(0,y)\sim(2\pi,y)$.   
Without loss of generality, we can assume $S=\{(\pi, y): \, y\in [\pi/2, 3\pi/2]\}$.  Then 
the metric completion of $M^2\setminus S$ is 
\be
	M_\infty = \frac{S^1 \times S^1 \times \{0\} \sqcup  S^1 \times S^1 \times \{1\}}{\sim},
\ee
where,
\be
(x,y,0) \sim (x,y.1) \;\;\text{for}\;\; (x,y) \not \in S,
\ee
with the distance $d_\infty$ given by
\be
	d_\infty \left([x,y,l] , [x',y',l']\right) = \lim_{\delta \to 0^+} d_{M^2\setminus S}\left( \left(x + (-1)^l \delta \;,\; y\right) , \left(x' + (-1)^{l'}\delta,y'\right)\right) 
\ee
for $l, l'=0,1$.  In particular,
\be
	d_\infty \left([\pi,\pi,0] , [\pi,\pi,1]\right) = \pi.
\ee

Notice that $M_\infty$ is not a manifold (not even Hausdorff as $B_r ([\pi , \pi , 0]) \cap B_r ([\pi , \pi , 1]) \neq \emptyset$ for all $r,r'$. ) Taking the connected precompact exhaustion 
\be
W_j = M^2 \setminus \left( [\pi/2 - 1/j , 3\pi/2 + 1/j] \times [\pi - 1/j , \pi + 1/j]\right),
\ee
we observe that
\begin{eqnarray}
	&\diam_{M_i}(W_j) \le  \diam (M^2) \notag\\ &\vol (M_i) = \vol (M^2) \\  &\vol_{g_i}(\partial W_j) \le 2\pi + 4,\notag
\end{eqnarray}
are uniformly bounded, also 
\be
	\lim_{j \to \infty} \vol_{g_i}(N \setminus W_j) =\lim_{j \to \infty} (2/j) (\pi + 2/j) = 0,
\ee
but, 
\begin{eqnarray}
	\lambda_{i,j,k} &=& \sup_{x,y\in W_j} |d_{(W_{k}, g_i)}(x,y)- d_{(M,g_i)}(x,y)| \notag \\
				& \ge& |d_{(W_{k}, g_i)} \left( (\pi - 1/j , \pi) , (\pi + 1/j , \pi) \right) - d_{(M,g_i)} \left( (\pi - 1/j , \pi) , (\pi + 1/j , \pi) \right)|\\
				&\ge&  \pi + 2/j + 2/k, \notag
\end{eqnarray}

Therefore,
\be
	\lim_{j\to\infty} \limsup_{i \to \infty} \limsup_{k\to\infty} \lambda_{i,j,k} \ge \lim_{j\to\infty} \limsup_{i \to \infty} \limsup_{k\to\infty} (\pi + 2/j + 2/k) = \pi,
\ee

\end{proof}

\begin{example}\label{ex-pulled-torus}
Let $(M^2,g_0)$ be the standard flat torus with $S$ as in Example~\ref{ex-slit-torus}.
Let $W_j= T_{1/j}(S)$ with respect to the flat norm.   Let
$g_j$ be the flat metric on $M^2\setminus W_j$. There exists smooth metrics $g_j$ on $M^2$ which agree with $g_0$ on $M^2 \setminus W_j$ such that 
the Gromov-Hausdorff and
Intrinsic Flat limits are the metric space created by taking the flat torus and
identifying all points in $S$ with each other. Then, $g_j$ converges smoothly away from $S$ to $g_\infty = g_0$. The metric completion of $(M \setminus S , g_\infty)$ is the slit torus as described in example \ref{ex-slit-torus}. These metrics
demonstrate that the diameter condition may not be replaced by
an extrinsic diameter condition in Theorem~\ref{diam-thm}
and in Theorem~\ref{c-diam-thm} but not the Ricci theorem
since they have negative curvature.
\end{example}

\begin{proof} 
Let $g_j= dt^2 + f_j(s,t)^2 ds^2$ where $f_j(s,t)=1$ on $W_j$
and $f_j(s,t)= 1/j$ on $S$, and smooth with values in $[1/j,1]$
everywhere. Let $\sim$ be defined as follows:
\be
	x \sim y  \;\;iff\;\;  x,y \in S
\ee

To estimate the GH and SWIF distance between $(M^2 , g_j)$ and $(\frac{M^2}{\sim} , d_0)$ we use the Theorem \ref{thm-subdiffeo}. First we need to find an estimate on the distortion $\lambda_j$ , which is defined by
\be \label{lambda-pulled-torus}
\lambda_j=\sup_{x,y \in W_j}
|d_{M_j}(x,y)-d_{\frac{M}{\sim}}(x,y)|.
\ee

Now let $P : M^2 \to \frac{M^2}{\sim}$ be the quotient map and Suppose $x_j,y_j \in \bar{W_j}$ achieve the maximum in the definition of $\lambda_j$ . Since $\frac{M^2}{\sim}$ is flat outside $\frac{S}{\sim}$, any shortest path, $\bar{C}_{x_j,y_j}$ , joining $x_j,y_j$ has to be a straight line. As a result,  $P^{-1}(\bar{C}_{x_j,y_j})$ is either the straight line , $C_{x_j,y_j}$ ,  in $M^2$ joining $x_j,y_j$ or the same straight line union the singular set $S$. And since the metric in $(M^2,g_j)$ is smaller than the flat metric on outside $W_j$ and coincide with the flat metric in $W_j$, we get
\begin{eqnarray}
	\lambda_j &\le& L(\bar{C}_{x_j,y_j}) - L(C_{x_j,y_j}) \\ &\le& \diam_{(M^2 , g_j)}(M^2 \setminus W_j) +  \diam_{\left(\frac{M^2}{\sim} , d \right)} \left(\frac{M^2 \setminus W_j}{\sim} \right).
\end{eqnarray}

Any two points in $M^2 \setminus W_j$ can be joined by a few horizontal segments, whose lengths add up to at most $2/j$ and a segment in $S$ with length less than $\pi/j$ and vertical segments, whose lengths add up to $2/j$ therefore, 
\be
	\diam_{(M^2 , g_j)}(M^2 \setminus W_j) \le \frac{\pi + 4}{j},
\ee
and projecting these segments by  $P$ we get
\be
	 \diam_{\left(\frac{M^2}{\sim} , d \right)} \left(\frac{M^2 \setminus W_j}{\sim}\right) \le 4/j,
\ee
hence,
\be
	\lambda_j \le \frac{\pi + 8}{j}  \to 0    \;\; \text{as} \;\; j \to \infty.
\ee

Now letting $\epsilon = 0$ in Theorem \ref{thm-subdiffeo}, we have a = 0  and
\be
	 \bar{h}_j = h_j = \sqrt{\lambda_j \left( \max \left\{\diam(W_j), \diam \left(\frac{W_j}{\sim} \right) \right\} + \lambda_j/4 \right) \,}  \to 0  \;\; \text{as} \;\;  j \to \infty.
\ee

So we conclude that
\be
d_{GH} \left(  (M^2,g_j) , \left(\frac{M^2}{\sim} , d \right) \right) \le a + 2\bar{h} +
\max\left\{ d^{M^2}_H(W_j, M^2), d^{\frac{M^2}{\sim}}_H \left(\frac{W_j}{\sim}, \frac{M^2}{\sim} \right)\right\} \to 0,
\ee
as $j \to \infty$ and also, it is easy to see that
\begin{eqnarray}
d_{\mathcal{F}}\left(  (M^2,g_j) , \left( \frac{M^2}{\sim} , d \right) \right) &\le&
\left(\bar{h} + a\right) \left(
\vol_2(W_j)+\vol_2 \left(\frac{W_j}{\sim} \right)+\vol_{1}(\partial W_j)+\vol_{1}\left(\frac{\partial W_j}{\sim} \right)\right) \notag \\
&&+\vol_2(M^2 \setminus W_j)	 + \vol_2 \left( \frac{M^2}{\sim} \setminus \frac{W_j}{\sim} \right) \to 0 \;\;\text{as}\;\; j \to \infty.
\end{eqnarray}

As we observed,  
\be
	 \lim_{j \to \infty}\diam_{(M^2,g_j)}(\partial W_j) \le \lim_{j\to\infty}\frac{\pi + 4}{j} = 0,
\ee
but,
\be
 \lim_{j \to \infty}\diam_{(W_j,g_j)}(\partial W_j) \ge \pi.
\ee
\end{proof} 

\subsection{Splines with Positive Scalar Curvature.}

In this section, we will present two examples that demonstrate that in our Theorems, the uniform lower Ricci curvature bound condition can not be replaced by a uniform scalar curvature bound. In the first example we construct a sequence of metrics on the $3$ - sphere which converge to the canonical sphere away from a singular point and also in the intrinsic flat sense but converges to a sphere with an interval attached to it in the Gromov-hausdorff sense. In the second example of this section, we will construct a sequence of metrics with positive scalar curvature which converge to the $3$ - sphere away from a singular point and also in the intrinsic flat sense while having no Gromov-hausdorff limit. The second example was in fact presented by Tom Ilmannen in a talk in 2004 at Columbia without details.  Both examples play an important role in \cite{SorWen2} however the fact that they have positive scalar curvature was never presented in detail in that paper.

\begin{lem}\label{positive-scalar-well}
	For any $L>0$ and $0<\delta<1$, there exists a smooth Riemannian metric on the $3$-sphere with positive scalar curvature which is obtained by properly gluing a spline of length $L + O(\delta^{\frac{1}{2}})$ and width $\le \delta$  to the unit $3$ - sphere.
\end{lem}

\begin{example}\label{ex-not-GH-scalar}
There are metrics $g_j$ on the sphere $M^3$
with positive scalar curvature such that
$M_j = (S^3, g_j)$ converge smoothly away from a point singularity
$S=\left\{p_0\right\}$ to the sphere, $S^3$, with $\diam(M_j) \le \pi + L + 2 $ and such that
    \be
            d_{\mathcal{F}}\left( M_j , S^3 \right) \; , \; d_{s\mathcal{F}}\left( M_j , S^3 \right) \to 0,
    \ee
and,
    \be
            d_{GH}\left( M_j , M_0 \right) \to 0,
    \ee
where $M_0 = S^3 \sqcup [0 , L]$ (the round sphere with an interval of length $L$ attached to it).
\end{example}

\begin{rmrk}\label{Ricci-bound-sharp}
	Example \ref{ex-not-GH-scalar} demonstrates that the uniform lower Ricci curvature bound condition in Theorem \ref{Ricci-thm-improved} and \ref{Ricci-diam-thm} can not be replaced by a uniform lower bound on the scalar curvature.
\end{rmrk}

\begin{example}\label{ex-many-tips-scalar}
There are metrics $g_j$ on the sphere $M^3$
with positive scalar curvature  such that
$M_j = (S^3, g_j)$ converge smoothly away from a point singularity
$S=\left\{p_0\right\}$ to the sphere, $S^3$, with $\diam(M_j) \le \pi + L + 2 $ and such that
    \be
            d_{\mathcal{F}}\left( M_j , S^3 \right) \; , \; d_{s\mathcal{F}}\left( M_j , S^3 \right) \to 0,
    \ee
and there is no Gromov-Hausdorff limit.
\end{example}

\begin{proof} of Lemma \ref{positive-scalar-well}.
The goal here is to attach a spline of finite length and arbitrary small width to a sphere whith positive scalar curvature. For this, we need to employ some ideas related to the Mass of rotationally symmetric manifolds. (c.f. ~\cite{LeeSormani1}). The construction goes as follows; we first find an admissible Hawking mass function (c.f. ~\cite{LeeSormani1} ) which will provide us with a three manifold embedded in $\E^4$ which is a hemisphere to which spline of finite length and small width is attached; we then, attach a hemisphere along its boundary.

Let $\delta < 1$ (this later will become the width of the spline). and let $r_{min} = 0$. Now we take an admissible Hawking mass function, $m_H(r)$ (which has to be smooth and increasing) that satisfies ($\epsilon$ to be determined later)
    \be
        m_H(r) =  r(1 - \epsilon^2)/2  \;  \text{for} \; r \in [0 , \delta^3],
    \ee
and,
    \be
        m_H(r) = r^3/2  \;  \text{for} \; r \in [\delta , 1].
    \ee
    
As in ~\cite{LeeSormani1}, define the function $z(r)$ via
   \be
        z (\bar{r}) = \int_{r_{min}}^{\bar{r}} \; \sqrt{\frac{2m_H(r)}{r - 2m_H(r)}} \; \mathrm{d}r.
   \ee
Note that $z$ depend on $\delta$.

$z(r)$ is unique up to a constant and gives our desired three manifold as a graph over $\E^3$. By our choice of $m_H(r)$ we get,
    \be \label{eq-thin-well}
        z' (r) = \sqrt{\frac{1 - \epsilon^2}{\epsilon^2}} \; \text{for} \; r \in [0 , \delta^3],
    \ee
and,
    \be \label{eq-neck}
        z' (r) = \sqrt{\frac{r^2}{1 -r^2}} \; \text{for} \; r \in [\delta , 1],
    \ee
and, since
     \be
             \frac{\delta^3}{2} (1 -\epsilon^2)   \le m_H(r) \le \delta^3/2 \; \text{for} \; r \in [\delta^3 , \delta],
     \ee
one obtains
    \be
         \sqrt{\frac{\delta^3 (1 - \epsilon^2)}{r - \delta^3 (1 - \epsilon^2)}}  \le  z' (r) \le \sqrt{\frac{\delta^3}{r - \delta^3}} \; \text{for} \; r \in [\delta^3 , \delta].
    \ee

Now, choose the $\epsilon$ that solves
   \be
        \delta^3 \sqrt{\frac{1 - \epsilon^2}{\epsilon^2}} = L,
   \ee
   
For some fixed $L$. From (\ref{eq-thin-well}), we have,
   \begin{eqnarray}
             \bar{L}(\delta) = z(\delta) - z(\delta^3)  \le \int_\delta^{\delta^3} \; \sqrt{\frac{\delta^3}{r - \delta^3}} = 2 \delta^{3/2} \left( \delta - \delta^3 \right)^{1/2} < 2,
   \end{eqnarray}
which goes to $0$ as $\delta$ goes to $0$.

We also get
   \be
             z(\delta^3) - z(0)  = \delta^3 \sqrt{\frac{1 - \epsilon^2}{\epsilon^2}} = L.
   \ee
   
The metric in terms of the distance from the pole, can be written as
    \be
        \bar{g}_\delta = ds^2 + f^2(s) g_{S^2} = (1 + [z'(r)]^2)dr^2 + r^2 g_{S^2}.
    \ee
    
In the virtue of the Theorem 5.4 in ~\cite{LeeSormani1}, we know that when $r \in [\delta , 1]$, we are on a unit sphere, and since
   \be
        \lim_{r \to 1^-} z'(r) = \infty,
   \ee
we have
   \be
        \lim_{r \to 1^-} f'(s) =   \lim_{r \to 1^-} r'(s) = \lim_{r \to 1^-} \frac{1}{\sqrt{1 + [z'(r)]^2}} = 0.
   \ee
   
Therefore, the boundary $r = 1$ is in fact a great $2$-sphere along which we can smoothly attach a $3$ - hemisphere. as follows

So far we have got the metric
  \be
       \bar{g}_\delta = (1 + [z'(r)]^2)dr^2 + r^2 g_{S^2} \;  \text{for} \;  r \in [0,1].
  \ee
  
Letting $r = \sin (\rho)$, one sees that
  \be
       \bar{g}_\delta = (1 + [z_\delta'(\sin(\rho))]^2) \cos^2 (\rho )d\rho^2 + \sin^2 (\rho) g_{S^2} \;  \text{for} \;  \rho \in [0, \pi/2].
  \ee
  
Therefore, on the sphere we define $g_\delta$ to be
   \be
(1 + [z_\delta'(\sin(\rho))]^2) \cos^2 (\rho )d\rho^2 + \sin^2 (\rho) g_{S^2} \;  \text{for} \;  \rho \in [0, \pi/2],
   \ee
and,
    \be
              d\rho^2 + \sin^2 (\rho) g_{S^2} \;  \text{for} \;  \rho \in [\pi/2, \pi],
    \ee
which has positive scalar curvature when $\rho \le \pi/2$ because it is isometric to $\bar{g}_\delta$ and has positive scalar curvature when $\rho \ge \pi/2$ because it is isometric to a round hemisphere. $g_\delta$ is smooth at $\rho = \pi/2$ because by ~\ref{eq-neck} near $\rho = \pi/2$,
     \be
            z_\delta' (\sin(\rho)) = \sqrt{\frac{\sin^2 (\rho)}{1 - \sin^2(\rho)}} = \tan(\rho).
     \ee
     
So,
    \begin{eqnarray}
     (1 + [z_\delta'(\sin(\rho))]^2) \cos^2 (\rho ) d\rho^2 + \sin^2 (\rho) g_{S^2} &=& (1 + \tan^2(\rho)) \cos^2 (\rho) + \sin^2 (\rho) g_{S^2}  \notag \\ &=& d\rho^2 + \sin^2(\rho) g_{S^2}.
    \end{eqnarray}
    
The key idea is that, using this method, one can attach symmetric spline of length $L + \bar{L}(\delta)$ and arbitrary small width $\delta < 1$ to a sphere while keeping the scalar curvature positive and $\diam(M_j) \le \pi + L + 2$. And the metric found can actually be written as a warped metric.

\end{proof}

\begin{proof} of Example \ref{ex-not-GH-scalar}.
Now let $\delta_j \to 0$ and take the sequence $M_j = (S^3 , g_{\delta_j})$ , where $g_{\delta_j}$ is given by the above construction for $\delta_j$. we are going to prove that $M_j$ converges to $M_0$ in Gromov-Hausdorff sense where $M_0$ is the unit three sphere to which an interval of length $L$ is attached; and $M_j$ converges to $S^3$ in intrinsic flat sense.

First notice that $M_j$ contains a subdomain $U_j$ which is isometric to $ U'_j = S^3 \setminus B_p (\arcsin( \delta_j))$ also letting $V_j = M_j \setminus U_j$ and $V'_j = S^3 \setminus U'_j$ one observes that since
    \begin{eqnarray}
            \vol(V_j) &\le& \int_0^\delta \; (4\pi r^2) (1 + [z'(r)]^2)^{1/2} \; \mathrm{d}r \notag \\ &\le& \int_0^\delta \; (4\pi r^2) (1 + |z'(r)|)\mathrm{d}r \\  &\le& (4\pi \delta^2) \left( \delta + L + \bar{L}(\delta) \right), \notag
    \end{eqnarray}
one gets $\vol (V_j) \to 0$ as $\delta_j \to 0$. Also it is obvious that $\vol(V'_j) \to 0$ as $\delta_j \to 0$.

Now, to be able to use Theorem ~\ref{thm-subdiffeo}, we need an estimate on
\be
    \lambda_j=\sup_{x,y \in U_j}
|d_{M_j}(x,y) - d_{S^3}(x,y)|.
\ee

Let $x,y \in U_j$ and let $\gamma$ and $c_{x,y}$ be the minimizing geodesic connecting $x$ and $y$ in $(M_j , g_j)$ and $S^3$ (resp.). If $\gamma$ lies completely in $U_j$, then, so does $c_{x,y}$ and $\gamma = c_{x,y}$ hence, $d_{M_j}(x,y) = d_{S^3}(x,y)$. If $\gamma \not \subset U_j$, therefore $\gamma = \gamma_1 + \gamma_2 + \gamma_3$ where $x \in \gamma_1$, $y \in \gamma_3$ and $\gamma_1, \gamma_3 \subset U_j$ and $\gamma_2 \subset V_j$ .  We are in either of the following cases

Case I: $c_{x,y} \subset U_j$

Obviously $L(\gamma_2) \le 2\pi\delta_j$, also we have
\be
    |d_{S^3}(x,p) - L(\gamma_1)| \le  \arcsin{\delta_j},
\ee
and
\be
     |d_{S^3}(y,p) - L(\gamma_3)| \le  \arcsin{\delta_j}.
\ee

Since $\delta_j \to 0$, for $j$ large enough,
\be
    L(\gamma) \approx d_{S^3}(x,p) + d_{S^3}(y,p)  > L(c_{x,y}),
\ee
which is a contradiction.

Case II: $c_{x,y} \not \subset U_j$.

Let $c_{x,y} = c_1 + c_2 + c_3$ where $x \in c_1$, $y \in c_3$ and $c_1, c_3 \subset U_j$ and $c_2 \subset V'_j$, Then, $L(c_2) \le 2 \arcsin(\delta_j)$ and also
\be
    |L(\gamma_i) - L(c_i)| \le 2\arcsin(\delta_j).
\ee

Therefore,
\be
    |d_{M_j}(x,y) - d_{S^3}(x,y)|=|L(\gamma) - L(c_{x,y})| \le 4\arcsin(\delta_j) + 2 \pi \delta_j.
\ee

This argument shows that $\lambda_j \to 0$ as $j \to \infty$. Since the intrinsic diameter $D_{U_j} \le \pi$ (both in $M_j$ and $S^3$), $h_j$ in Theorem ~\ref{thm-subdiffeo} goes to $0$ as $j \to \infty$. Letting $\epsilon = 0$ in  Theorem ~\ref{thm-subdiffeo} we get $a=0$ and $\bar{h}_j = h_j$ therefore,
   \be
        d_{\mathcal{F}} (M_j , S^3) \le \bar{h_j} \left( 2 \vol(U_j) + 2 \vol(\partial U_j)\right)    + \vol(V_j) + \vol(V'_j).
   \ee
which gives
        \be
             d_{\mathcal{F}} (M_j , S^3) \to 0   \; \text{as} \; j \to \infty.
        \ee

To prove that $M_j$ converges to $M_0$ in Gromov-hausdorff sense, we will estimate $d_{GH}(M_j , M_0)$ using the fact that
        \be
                d_{GH}(M_j,M_0) = \frac{1}{2} \inf_{\mathfrak{R}}\left( \operatorname{dis} \mathfrak{R} \right),
        \ee
where, the infimum is over all correspondences $\mathfrak{R}$ between $M_j$ and $M_0$ and $\operatorname{dis} \mathfrak{R}$ is the distortion of $\mathfrak{R}$ given by
       \be
                \operatorname{dis} \mathfrak{R} = \sup \left\{ |d_{M_j}(x , x')| - |d_{M_0} (y , y')| : (x,y) , (x',y') \in \mathfrak{R}\right\}.
       \ee
       
For details see ~\cite[p. 257]{BBI}.

We need to find correspondences $\mathfrak{R}_j$ between $M_j$ and $M_0$ such that $\operatorname{dis} \mathfrak{R}_j  \to 0$ as $j \to \infty$. Consider $W_j \subset \E^4$ given by
       \be
             W_j = M_j \cup M_0 = M_0  \cup V_j  = M_j \cup \left( M_0 \setminus U_j \right).
       \ee
       
In fact, we can picture $W_j$ as the union of sphere, an interval of length $L$ and a spline of length $L + \bar{L}(\delta_j)$ around the spline. Then we have, $M_j \subset W_j$ and $M_0 \subset W_j$ define the following surjective maps $f : W_j \to M_j$ and $ g: W_j \to M_0$
        \be
                f |_{M_j} = \operatorname{id}.
        \ee
        \be
                f (w) = (r \left(z(w) \right), 0 , 0 , z(y)) \in V_j  \; \text{for} \; w \in \left( M_0 \setminus U_j \right).
        \ee
        
Note that when $w \in \left( M_0 \setminus U_j \right)$, $f(w)$ is the point in $V_j \cap xz-\text{plane}$ closest to $w$ .

Similarly let
      \be
                g |_{M_0} = \operatorname{id}.
        \ee
        \be
                g (w) = \text{the point in $\left( M_0 \setminus U_j \right)$ closest to $w$}    \; \text{for} \; w \in V_j.
        \ee
        
Let $\mathfrak{R}_j$ be the following  correspondence between $M_j$ and $M_0$,
        \be
                \mathfrak{R}_j = \{ \left( f(w) , g(w) \right) : w \in W_j \}.
        \ee
        
\textbf{Claim:} $\operatorname{dis} \mathfrak{R}_j \to 0$ as $ j \to \infty $.

    Pick $w_1 , w_2 \in W_j$, and suppose $\gamma + \lambda$ is the minimal geodesic in $M_j$ connecting $f(w_1)$ and $f(w_2)$ where $\gamma \subset U_j$ and $\lambda \subset V_j$. and let $\gamma' + \lambda'$ be the (possibly) broken minimal geodesic connecting $g(w_1)$ and $g(w_2)$ where, $\gamma' \subset U_j$ and $\lambda' \subset \left( M_0 \setminus U_j \right)$. Without loss of generality we assume that $z(w_1) \le z(w_2)$ . Next we need estimates on the lengths of $\gamma , \lambda , \gamma' , \lambda'$. Let $q$ and $q'$ be starting points on $\lambda$ and $\lambda'$ respectively, then
        \be\label{eq-claim-1}
               \int_{q}^{f(w_2)} \; \mathrm{d}z \le \int_{q}^{f(w_2)} \; s'(z) \; \mathrm{d}z \le L(\lambda) \le  \int_{q}^{f(w_2)} \; s'(z) \; \mathrm{d}z + 2\pi \delta_j,
        \ee
where, the term $2\pi \delta_j$ is the maximum perimeter of the well and note that any two point on the we can be joined by a radial geodesic followed by a curve of length less than  $2\pi \delta_j$.

Since $ds^2 = \left( 1  + [r'(z)]^2 \right) dz^2$ we get
        \begin{eqnarray}
                 \int_{q}^{f(w_2)} \; s'(z) \; \mathrm{d}z &\le& \int_{q}^{f(w_2)} \mathrm{d}z + \int_{q}^{f(w_2)} |r'(z)| \mathrm{d}z \\ &\le& \int_{q}^{f(w_2)} \mathrm{d}z + \delta_j \left(L + \bar{L}(\delta_j)\right). \notag
        \end{eqnarray}
        
We also have
        \begin{eqnarray}
               \int_{q'}^{g(w_2)} \; \mathrm{d}z \le L(\lambda') \le  \int_{q'}^{g(w_2)} \; \mathrm{d}z + 2\arcsin(\delta_j).
        \end{eqnarray}
        
The last inequality comes from the fact that any two points in $M_0 \setminus U_j$  can be joined by a (broken) geodesic which is a straight line followed by a curve of length at most $\diam(V'_j) = 2\arcsin(\delta_j)$ .

On the other hand by our construction
        \be
                |  \int_{q}^{q'} \; \mathrm{d}z | \le \bar{L}(\delta_j),
        \ee
and,
        \be
                |  \int_{f(w_2)}^{g(w_2)} \; \mathrm{d}z | \le \bar{L}(\delta_j).
        \ee
        
Therefore,
  \be\label{eq-claim-2}
                |  \int_{q}^{f(w_2)} \; \mathrm{d}z  -  \int_{q'}^{g(w_2)} \; \mathrm{d}z | \le 2 \bar{L}(\delta_j).
        \ee
        
From   \ref{eq-claim-1} - \ref{eq-claim-2}, we get
        \be\label{eq-claim-3}
                | L(\lambda) - L(\lambda') | \le \delta_j \left(L + \bar{L}(\delta_j) + 2\pi\right) + 2\bar{L}(\delta_j) + 2\arcsin(\delta_j).
        \ee
        
Also one observes that when $w_1 \in U_j$, then $\gamma$ and $\gamma'$ are geodesics on the sphere starting from the same point and ending up in $V'_j$ which means that
        \be\label{eq-claim-4}
               | L(\gamma) - L(\gamma') | \le \diam(V'_j) = 2 \arcsin(\delta_j).
        \ee
        
From (\ref{eq-claim-3}) and (\ref{eq-claim-4}),
        \begin{eqnarray}
               |d_{M_j} \left(f(w_1) , f(w_2) \right) - d_{M_0} \left(g(w_1) , g(w_2) \right)| &\le& | L(\gamma) - L(\gamma') | + | L(\lambda) - L(\lambda') | \notag \\ &\le& \delta_j \left(L + \bar{L}(\delta_j) + 2\pi\right) \\ &&+ 2\bar{L}(\delta_j) + 4\arcsin(\delta_j). \notag
        \end{eqnarray}
        
Therefore,
        \be
                \operatorname{dis} \mathfrak{R}_j \le \delta_j \left(L + \bar{L}(\delta_j) + 2\pi\right) + 2\bar{L}(\delta_j) + 4\arcsin(\delta_j),
        \ee
which shows that  $ \operatorname{dis} \mathfrak{R}_j \to 0 $ as $ j \to \infty$ .   This completes the proof of the claim.

To prove that the convergence off the singular set $S = \lbrace p_0 \rbrace$, which is the bottom of the well, let $\rho_0 > 0$, then
        \be
             g_{\delta_j}  \to g_{S^3} \; \text{on} \; \rho^{-1} \left( [\rho_0 , \pi] \right),
        \ee
because for $j$ sufficiently large, $\delta_j < \rho_0$, which by our construction means that $g_{\delta_j} = g_{S^3}$ on $\rho^{-1} \left( [\rho_0 , \pi] \right)$.
\end{proof}

\begin{proof} of Example \ref{ex-many-tips-scalar}.
Let $g (p_0 , s)$ denote a symmetric spline of length L centered at the point $p_0$ with width $s$ (as constructed in the previous examples). Also fix a great circle and a point $0$ in $S^2$, so now, when we say a point given by the angle $\theta$, it means a point on this great circle given by the angle $\theta$. On the round sphere, $r = \pi - \frac{1}{2^j}$ is a $2$-sphere with radius $\sin(\pi - \frac{1}{2^j})$, Therefore, balls with radius $ s_j = \frac{1}{j} \sin(\pi - \frac{1}{2^j}) \sqrt{2-2\cos(\frac{2\pi}{2^j})}$ centered at points $p_k$ given by the angle $\theta=\frac{2k\pi}{2^j}$ are disjoint. Now for each $j$, we can glue metrics $g(p_k , s_j )$ which agree with the metric on the spline given
in Example~\ref{ex-not-GH-scalar}.   Outside of each $B_{p_k}(s_j)$, we
set $g_j=g_0$.  It is easy to see that by our construction, $g_j$ agrees with the round metric for $r < \pi - \frac{1}{2^j} - s_j$ and has $2^j$ splines of length $L$ and width $s_j$ and also the volume of the non spherical part is going to $0$ as $j \to \infty$. So by taking $U_j = r^{-1}([0 , \pi - \frac{1}{2^j} - s_j ])$, we see that again all conditions in Theorem ~\ref{flat-to-settled} are satisfied therefore we have the flat convergence to the settled completion. 
\end{proof}


\section{Applications}\label{sec-applications}

\subsection{Background}

In this section, we will briefly outline one the applications of our Theorems which is proving a conjecture of Candelas and de la Ossa for conifold flops and transitions. Different versions of this conjecture have been proved in \cite{Rong-Zhang} and \cite{Song-Candelas-conj}. Both their proofs require advanced PDE methods and Nash-Moser iteration to get uniform $C^{k,\alpha}$ bouds on the potential functions. Our approach on the other hand only relies on our main theorems. 

Let $M_0$ be a singular normal projective variety of dimension $n$ with singular set $S$. One way to disingularize $M_0$ is called the \emph{resolution} of singularities of $M_0$ the defintition of which is as follows:

\begin{defn}[Resolution]
	A pair $\left( \bar{M} , \bar{\pi} \right)$ of a projective manifold $\bar{M}$ and a morphism $\bar{\pi} : \bar{M} \to M_0$ is called a resolution of $M_0$ if $\bar{\pi}: \bar{M}\setminus \bar{\pi}^{-1}(S) \to M_0 \setminus S$ is biholomorphic.
\end{defn}

Another usual method of disingularization is called \emph{smoothing} and defined as follows

\begin{defn}[Smoothing]
 	Let $D \subset \C$ be the unit disc. A pair $\left( \mathcal{M} , \pi \right)$ of a projective $n+1-$dimensional manifold $\mathcal{M}$ and proper flat morphism $\pi : \mathcal{M} \to D$  is called a smoothing if $\pi^{-1}(0) = M_0$ and $\pi^{-1}(t) = M_t$ is a smooth projective $n-$dimensional manifold for $t \in D\setminus\{0\}$.
\end{defn}

\begin{defn}\label{defn-conifold}[Conifold]
	A normal variety $M_0$ with $\dim_{\C}(M_0) = n$ is called a conifold if its singularities are all ordinary double points i.e. any singular point is locally given by 
	\be
		z_0^2 + z_1^2 + \dots + z_n^2 = 0.
	\ee
\end{defn}

\begin{defn}[Extremal and Conifold Transitions]
If $\left( \bar{M} , \bar{\pi} \right)$ and $\left( \mathcal{M} , \pi \right)$ are a resolution and smoothing of $M_0$ respectively, then, the diagram
\be
	\bar{M} \to M_0 \leadsto M_t   \;\;\;  (t \neq 0),
\ee
denotes the process of going from $\bar{M}$ to $M_t$ and is called an extremal transition. Whenever $M_0$ is a conifold, the transition is called a conifold transition.
\end{defn}

\begin{defn}[Flops]
	If $\left( \bar{M}_1 , \bar{\pi}_1 \right)$ and $\left( \bar{M}_2 , \bar{\pi}_2 \right)$ are two resolutions of $M_0$, then the process of going from $\bar{M}_1$ to $\bar{M}_2$ is called a flop and is denoted by the following diagram
	\be
		\bar{M}_1 \to M_0 \dashrightarrow \bar{M}_2.
	\ee
If $M_0$ is a conifold, then the flop is called a conifold flop. 
\end{defn}

Here, we will mention a mathematical formulation of the Candelas and de la Ossa’s conjecture due to Rong-Zhang~\cite{Rong-Zhang}:

\begin{conj}[Candelas - de la Ossa]\label{conj-candelas}
	Let $M_0$ be a singular $n-$dimensional normal variety then,\\
	
	\textbf{(i)} Extremal transitions $\bar{M} \to M_0 \leadsto M_t   \;\;\;  (t \neq 0)$ are continuous with respect to the Gromov-Hausdorff distance i.e. there exist families of Ricci-flat K\"{a}ler metics $\bar{g}_s$ on $\bar{M}$ and $g_t$ on $M_t$ and compact metric space $\left( X , d_X \right)$ such that
	\be
		\left( \bar{M} , \bar{g}_s \right) \xrightarrow{G-H} \left( X , d_X \right)\xleftarrow{G-H} \left( M_t , g_t\right)
	\ee\\
	
	\textbf{(ii)}  Flops $\bar{M}_1 \to M_0 \dashrightarrow \bar{M}_2$ are continuous with respect to the Gromov-Hausdorff distance i.e. there families of Ricci-flat K\"{a}hler metrics $\bar{g}_{i,s}$ and a compact metric space $\left( X , d_X \right) $ such that
	
	\be
		\left( \bar{M}_1 , \bar{g}_{1,s} \right) \xrightarrow{G-H} \left( X , d_X \right)\xleftarrow{G-H} \left( \bar{M}_2 , \bar{g}_{2,s}\right)
	\ee
\end{conj}


\subsection{Generalized Calabi-Yau Theorem}

Eyssidieux-Guedj-Zeriahi in \cite{EGZ} prove a generalized Clabi-Yau theorem which goes as follows:

\begin{prop}[Generalized Calabi-Yau Theorem]
	Let $M_0$ be a Calabi-Yau variety, then to any ample line bundle $\mathcal{L}_0$ there corresponds a unique Ricci-flat K\"{a}hler form $\omega \in c_1(\mathcal{L}_0)$.
\end{prop}


\subsection{Smooth Convergence and Diameter Bounds}

Here, we will briefly mention two theorems that establish the smooth convergence away from the singular set $S$. 

Tosatti in \cite{Tosatti-KE} proved a convergence result for Calabi-Yau metrics which can be paraphrased as follows:

\begin{prop}[Tossati\cite{Tosatti-KE}]\label{prop-resolution-convergence}
	Suppose the Calabi-Yau variety $M_0$ admits a crepant rsolution $\left(\bar{M} , \bar{\pi} \right)$ and an ample lie bundle $\mathcal{L}_0$. Let $\beta_s \in H^{1,1}\left(\bar{M} , \R \right)$ , $s \in (0,1)$ be a family of K\"{ahler} classes such that
	\be
		\lim_{s \to 0} \beta_s = \bar{\pi}^* c_1 \left( \mathcal{L}_0\right).
	\ee
Let $\bar{g}_s$ be the unique Ricci-flat K\"{a}hler corresponding to the K\"{a}hler form $\bar{\omega}_s \in \beta_s$ then, as $s \to 0$,
\be
	\bar{g}_s \xrightarrow{C^\infty} \bar{\pi}^*g \;\;\; and \;\;\; \bar{\omega}_s \xrightarrow{C^\infty} \bar{\pi}^*\omega
\ee
on any compact subset $K \subset\subset \bar{M} \setminus \bar{\pi}^{-1}(S)$.
\end{prop}

On the other hand Ruan-Zhang\cite{Rong-Zhang} proved convergence (off the singular set)  and diameter bound rasults for smoothings of a Calabi-Yau conifold. Although their result is sufficient for our purpose, we will mention a generalization of convergence results in \cite{Ruan-Zhang} due to Rong-Zhang \cite{Rong-Zhang} for the sake of being thorough.

\begin{prop}[Rong-Zhang\cite{Rong-Zhang}]\label{prop-smoothing-convergence}
	Let $M_0$ ba a Calabi-Yau variety with $\dim_\C(M_0) = n$ ($n \leq 2$) with singular set $S$. Assume that $M_0$ admits a smoothing $\pi: \mathcal{M} \to D$ and an ample line bundle $\mathcal{L}$ whose relative canonical bundle $\mathcal{K}_{\mathcal{M}/D} \cong \mathcal{O}_{\mathcal{M}}$. Let $g_t$ be the unique Ricci-flat K\"{a}hler metric with K\"{a}hler form $\omega_t \in c_1 (\mathcal{L})|_{M_t} \in H^{1,1}\left(M_t  , \R \right) \;\;(t \neq 0)$ , and $\omega$ be the unique singular Ricci-flat K\"{a}hler form on $M_0$ with $\omega \in c_1(\mathcal{L})|_{M_0} \in \left( M_0 , \mathcal{PH}_{M_0} \right)$, then
	\be
		F_t^*g_t \xrightarrow{C^\infty} g \;\;\; and \;\;\; F_t^*\omega_t \xrightarrow{C^\infty} \omega
	\ee
on any compact subset $K \subset \subset M_0 \setminus S$, where $F_t: M_0 \setminus S \to M_t$ is a smooth family of embeddings and $g$ is the corresponding K\"{a}hler metric of $\omega$ on $M_0 \setminus S$. Furtheremore, the diameter of $\left( M_t , g_t \right) \;\; (t \neq 0)$ enjoys the uniform upper bound
\be
	\diam_{g_t} \left( M_t \right) \le D
\ee 
where $D$ is a constant independent of $t$.
\end{prop}



\subsection{Good Exhaustion of the Regular Set}

In this section, we will consider the case of finitely many ordinary double point singularities $S = \left\{ p_1 , p_2 , \dots , p_m \right\} \subset M$ and will give a very natural exhaustion, $\left\{ K_i \right\}$, of the regular set $M \setminus S$ that satisfy all the hypotheses in Theorem \ref{Ricci-thm-improved}

Recall from Definition \ref{defn-conifold} that a double point singularity is locally modelled as $p = ( 0 , 0 , \dots , 0)$ in $\left\{z \in \C^{n+1} \; | \; z_0^2 + z_1^2 + \dots + z_n^2 = 0 \right\}$.

We construct the exhaustion by just taking out metric balls around singular points , i.e. let $r_i$ be a sequence of positive numbers such that $\lim_{i \to \infty} r_i = 0$ and let
\be
	W_i = M_0 \setminus \bigcup_j B\left( p_j , r_i\right)
\ee

At this point we recall the Bishop-Gromov relative volume comparison theorems which will show that we have chsen a good exhaustion. 

\begin{thm}[Bishop-Gromov Relative Volume Comparison]\label{thm-Bishop-Gromov}
	Suppose $M^n$ has $\Ricci_M \ge (n-1)H$ then,

\be
	\frac{\vol \Big( \partial B(p,r) \Big)}{\vol_H \Big( \partial B(r) \Big)} \;\; and \;\; \frac{\vol \Big( B(p,r) \Big)}{\vol_H \Big( B(r) \Big)}
\ee

are nonincreasing in $r$.
 
\end{thm}

We also need the following well-known expansion of volume ratio for Riemannian maniolds:

\begin{prop}\label{prop-volume-ratio}

For any Riemannian manifold $M$ , we have:

\be\label{}
	\frac{\vol \Big(  B(\epsilon , p) \subset M \Big)}{\vol \Big( B(\epsilon) \subset \R^n  \Big)} = 1 - \frac{Scal}{6(n+2)}\epsilon^2  + O(\epsilon^4),
\ee
and,

\be
	\frac{\vol \Big(  \partial B(\epsilon , p) \subset M \Big)}{\vol \Big( \partial B(\epsilon) \subset \R^n  \Big)} = 1 - \frac{Scal}{6n}\epsilon^2  + O(\epsilon^4)
\ee
\end{prop}

Appying Theorem \ref{thm-Bishop-Gromov} and Proposition \ref{prop-volume-ratio} to our Ricci flat metrics prove that the exhaustion we have chosen and their boundaries satisfy the volume and diameter conditions in Theorems \ref{Ricci-thm-improved} and \ref{diam-thm} namely,

\be \label{area-2}
\vol_{g_i}(\partial W_j) \le A_0,
\ee
and
\be \label{not-vol-2}
\vol_{g_i}(M\setminus W_j) \le V_j \textrm{ where } \lim_{j\to\infty}V_j=0,
\ee

\begin{rmrk}
We are allowed to use the comparison theorems because we are in the setting of K\"{a}hler manifolds.
\end{rmrk}

\subsection{Proof of Theorem \ref{thm-Candelas}}

Propositions \ref{prop-resolution-convergence} and \ref{prop-smoothing-convergence} provide the diameter bounds and convergence off the singular set assumptions in Theorems \ref{Ricci-thm-improved} and \ref{diam-thm} and by using Theorem \ref{thm-Bishop-Gromov} and Proposition \ref{prop-volume-ratio} we get the uniform volume bounds for the exhaustion $W_j$ and their boundaries as are required in Theorems \ref{Ricci-thm-improved} and \ref{diam-thm}.


\bibliographystyle{plain}
\bibliography{reference2013}


\end{document}